\documentclass{amsart}
\usepackage{amsfonts}
\usepackage{amsthm}
\usepackage{amssymb}
\usepackage{amsmath}
\usepackage{upref}
\usepackage{mathrsfs}
\usepackage{color}
\usepackage[citecolor=blue,colorlinks=true]{hyperref}
\usepackage{enumitem}
\usepackage{marginnote}

\newtheorem{thm}{Theorem}[section]
\newtheorem{lemma}[thm]{Lemma}
\newtheorem{prop}[thm]{Proposition}
\newtheorem{cor}[thm]{Corollary}
\theoremstyle{definition}
\newtheorem{defin}[thm]{Definition}
\theoremstyle{remark}
\newtheorem{rem}[thm]{Remark}

\numberwithin{equation}{section}

\newcommand{\dif}{\mathrm{d}}
\newcommand{\totdif}{\mathrm{D}}
\newcommand{\mf}{\mathscr{F}}

\newcommand{\mr}{\mathbb{R}}
\newcommand{\prst}{\mathbb{P}}
\newcommand{\stred}{\mathbb{E}}
\newcommand{\ind}{\mathbf{1}}
\newcommand{\mn}{\mathbb{N}}

\newcommand{\mt}{\mathbb{T}^N}
\newcommand{\mttwo}{\mathbb{T}^2}

\DeclareMathOperator{\supp}{supp}
\DeclareMathOperator{\trace}{tr}
\DeclareMathOperator{\cof}{cof}
\DeclareMathOperator{\diver}{div}
\DeclareMathOperator{\id}{Id}

\newcommand{\tec}{{\overset{\cdot}{}}}
\newcommand{\semicol}{;}
\newcommand{\sq}{H(\nabla u)}
\newcommand{\sqt}{H(\nabla \tilde u)}
\newcommand{\eps}{\varepsilon}

\newcommand{\R}{\mathbb{R}}
\newcommand{\N}{\mathbb{N}}

\newcommand{\kappav}{v}
\newcommand{\kappave}{v}

\begin{document}

\title[Stochastic Mean Curvature Flow for two-dimensional graphs]{Weak Solutions for a Stochastic Mean Curvature Flow of two-dimensional graphs}

\author{Martina Hofmanov\'a}
\address[M. Hofmanov\'a]{Max-Planck-Institut f\"ur Mathematik in den Naturwissenschaften, Inselstra\ss e 22, 04103 Leipzig, Germany\newline  Technische Universit\"at Berlin, Institut f\"ur Mathematik, Stra\ss e des 17. Juni 136,
10623 Berlin, Germany}
\email{hofmanov@math.tu-berlin.de}

\author{Matthias R\"oger}
\address[M. R\"oger]{Fakult\"at f\"ur Mathematik, Technische Universit\"at Dortmund, Vogelpothsweg 87,
44227 Dortmund, Germany}
\email{matthias.roeger@math.tu-dortmund.de}

\author{Max von Renesse}
\address[M. von Renesse]{Universit\"at Leipzig, Fakult\"at f\"ur Mathematik und Informatik, Augustusplatz 10, 04109 Leipzig, Germany}
\email{renesse@uni-leipzig.de}

\begin{abstract}
We study a stochastically perturbed mean curvature flow for graphs in $\R^3$ 
over the two-dimensional unit-cube subject to periodic boundary conditions. 
The stochastic perturbation is a one dimensional white noise acting 
uniformly in all points of the surface in normal direction.  We establish the 
existence of a weak martingale solution. The proof is based on energy methods 
and therefore presents an alternative to the stochastic viscosity solution 
approach. To overcome difficulties induced by the degeneracy of the mean 
curvature operator and the multiplicative gradient noise present in the model we 
employ a three step approximation scheme together with refined stochastic
compactness and martingale identification methods.

\end{abstract}


\date{\today}

\maketitle

\section{Introduction}
\noindent   Motion by mean curvature of embedded hypersurfaces in $\R^{N+1}$ is an important prototype of a geometric evolution law and has been intensively studied in the past decades, see for example the surveys \cite{MR1931534}, \cite{Ecke04}, \cite{Mant11} or \cite{Bell13}. Mean curvature flow is characterized as a steepest descent evolution for the surface area energy (with respect to an $L^2$ metric) and constitutes a fundamental relaxation dynamics for many problems where the interface size contributes to the systems energy. In
physics it arises for example as an asymptotic reduction of the Allen--Cahn model for the motion of phase boundaries in binary alloys \cite{AlCa79}. 

One of the main difficulties in the mathematical treatment of mean curvature flow is the appearance of topological changes and singularities in finite time, for example by the development of corners and a collapse of parts of the surfaces onto a line in the evolution of a thin dumbbell-shape surface in $\R^3$. Only in particular situations such events are excluded: in the case of initial surfaces given by entire graphs over $\R^{N}$ classical solutions exist for all times \cite{EcHu91}; initally smooth, compact, convex hypersurfaces become round and shrink to a point in finite time \cite{Huis90}.

In order to deal with singularity formation and topological changes generalized formulations have been developed. In his pioneering work Brakke \cite{Brak78} employed a geometric measure theory approach to obtain a general global in time existence result. Level set approaches and viscosity solutions were introduced by Evans and Spruck \cite{EvSp91,EvSp92,EvSp92a,EvSp95} and Chen, Giga, Goto \cite{ChGG91}. Evolutions beyond singularity formation and topological changes can also be obtained by De Giorgi's barrier method \cite{BePa95,BeNo97}, approximation by the Allen--Cahn equation \cite{EvSS92,Ilma93,BaSS93}, time-discretization \cite{LuSt95,AlTW93} and by elliptic regularization \cite{Ilma94}. Several of these approaches have been applied also to more general geometric evolution laws and for perturbations by various forcing terms.

Stochastic mean curvature flow was proposed in 
\cite{citeulike:2163102} as a refined model incorporating the
influence of thermal noise. As a result one may think of a random evolution $(M_t)_{t>0}$ of surfaces in $\R^{N+1}$ given by immersions $\phi_t:M\to\R^{N+1}$, where $M$ is a smooth manifold, and where the increments are given by   
\begin{equation}   
	\dif\phi_t(x)   =  \vec{H}(x,t) \dif t
	+ W(\nu(x,t),\phi_t(x),\circ \dif t), \quad x \in M, \label{initialmodel} 
\end{equation}
where $\vec{H}(x,t)$ denotes the mean curvature vector of $M_t$ in $\phi_t(x)$, $\nu(x,t)$ is the unit normal field on $M_t$ and $W: \mathbb{S}^N\times\R^{N+1} \times \R^+\rightarrow \R^{N+1}$ is a
model specific random field with $W(\theta,y,\circ \dif t)$ being  its 
Stratonovich differential (here one could even allow for an additional dependence of $W$ on $M_t$). As an example  consider   
$W(\theta,y,t)=
\theta\, \varphi(y) \beta_t  $ for $\varphi \in C^{\infty}(\R^{N+1})$ with a standard
real Brownian motion $\beta$, inducing the dynamics   
\begin{equation}   
	\dif\phi_t(x)    =  \nu(x,t) \bigl(\kappa(x,t) \dif t
				+  \varphi(\phi_t(x)) \circ \dif \beta_t\bigr)  \label{normalmodel},  
\end{equation}
where $\kappa(x,t):=\vec{H}(x,t)\cdot\nu(x,t)$ denotes the scalar mean curvature.
As in the deterministic case \eqref{normalmodel} can be formulated as a level set equation. Here the evolution of a function $f:\R^{N+1}\times\R^+\to\R$ is prescribed whose level sets all evolve according to \eqref{normalmodel}. This leads to a stochastic partial differential equation (SPDE) of the form 
\begin{equation}
  \dif f(x,t)  = |\nabla f|(x,t)\,  {\rm div}\bigg( \frac{\nabla
f}{ |\nabla f|} \bigg)(x,t)\, \dif t+ \varphi(x,f(x,t)) |\nabla f|(x,t)\,  \circ \dif \beta_t
\label{lsspde}.
\end{equation}
We stress that the choice of the Stratonovich differential instead of an It\^o term is necessary to retain the geometric meaning of the equation and to make it invariant under reparametrization of the level set function \cite{LiSo02}.

If we restrict ourselves to random evolutions of graphs, scalar mean curvature, normal vector, and velocity of  an evolution $u:\R^N\times \R^+\to\R$ and the associated graphs are given by 
\begin{align*}
	\kappa &={\rm div}\bigg( \frac{\nabla
u}{\sqrt {1 + |\nabla u|^2}} \bigg),\quad \nu=\frac{1}{\sqrt {1 + |\nabla u(x,t)|^2}}(-\nabla u,1)^T,\\
	 \dif\phi_t \cdot\nu &= \frac{1}{\sqrt {1 + |\nabla u(x,t)|^2}}\dif u.
\end{align*}
Equation \eqref{normalmodel} then reduces to the SPDE
\begin{align}
 	\dif u (x,t) &= \sqrt {1 + |\nabla u(x,t)|^2}\, {\rm div}\bigg( \frac{\nabla
u}{\sqrt {1 + |\nabla u|^2}} \bigg)\,(x,t) \dif t \notag\\
	&\qquad\qquad + \sqrt{1 +|\nabla u(x,t)|^2}\varphi(x,u(x))  \circ \dif \beta_t.
\label{grafspde}
\end{align}
Note that we naturally obtain the factor $\sqrt{1 +|\nabla u|^2}$ in front of the noise term and that \eqref{lsspde} reduces to \eqref{grafspde} for $f(x,y)=y-u(x)$, $(x,y)\in\R^N\times\R$. Vice versa, following the approach of Evans and Spruck \cite{EvSp91} one could approximate \eqref{lsspde} by a problem for rescaled graphs (in $\R^{N+2}$), which leads to an equation similar to \eqref{grafspde} but with $\sqrt {1 + |\nabla u|^2}$ replaced by $\sqrt {\varepsilon^2 + |\nabla u|^2}$, $\varepsilon>0$ a small parameter.

We further observe that the first term on the right-hand side of \eqref{grafspde} can be rewritten as 
\begin{align}
	\sqrt {1 + |\nabla u|^2}\, {\rm div}\bigg( \frac{\nabla
u}{\sqrt {1 + |\nabla u|^2}} \bigg) = \bigg(\id - \frac{\nabla
u}{\sqrt {1 + |\nabla u|^2}}\otimes \frac{\nabla
u}{\sqrt {1 + |\nabla u|^2}}\bigg):\totdif^2u \label{eq:H-term}
\end{align}
and that this term corresponds to a degenerate quasilinear elliptic differential operator of second order in the spatial variable.

Even though we circumvent problems with topological changes by restricting ourselves to graphs, substantial mathematical difficulties are still present in the stochastic case. Most importantly, one has to deal with the multiplicative noise with nonlinear gradient dependence and with the degeneracy in the quasilinear elliptic term, which makes a rigorous treatment challenging. In particular, a general well-posedness theory seems still to be missing.
Motivated by the deterministic counterpart of \eqref{lsspde} Lions and
Souganidis 
introduced a notion of stochastic viscosity solutions   
\cite{MR1647162,MR1659958,MR1799099,MR1807189}, but certain technical details of this approach
are still being investigated  \cite{MR1920103,MR2765508,MR3152786}. 
The model \eqref{normalmodel} with constant $\varphi
=\epsilon >0$ was also studied independently  in $N=1$  by
Souganidis and
Yip \cite{MR2037245} 
resp.\ Dirr, Luckhaus and Novaga \cite{MR1867935}, proving   a 'stochastic
selection
principle' for $\epsilon$ tending to zero\footnote{This
means for 
$\epsilon \to 0$  the  level sets the solutions $f_t^\epsilon$ to
\eqref{lsspde} converge a.s.\ to some 
solution of  mean curvature flow  even in cases when $f^0_t$ develops 
'fattening', i.e.\  has zero level sets
of positive Lebesgue measure.}.

\smallskip Several approaches to construct generalized solutions to 
other versions of
\eqref{initialmodel}  can be found in the literature, such as by  
Yip \cite{MR1656479} who selects subsequential
limits along tight approximations of a scheme that combines a time-discrete mean curvature flow and a stochastic flow of diffeomorphism of the ambient space. More recently, extending the rigorous analysis of the sharp interface
limit of
the 1-dimensional stochastic Allen-Cahn equation by Funaki
\cite{MR1337253} in \cite{zbMATH06217659} tightness of solutions for an Allen--Cahn equation perturbed by a stochastic flow was proved. However, both in \cite{MR1656479} and in \cite{zbMATH06217659} a characterization of the limiting evolution law has not been given. \nocite{MR2642385} 
Finally, it  was shown in \cite{MR2888287} that several variants of
\eqref{initialmodel} in dimension $1+1$ can be solved in the variational SPDE
framework (see also \cite{2014arXiv1405.5866G,MR3249580} for refinements resp.
numerical analysis), but this approach is not applicable in higher dimensions.
For completeness let us also mention that the analysis of associated
formal
large
deviation functionals was started in \cite{MR2284215} and remains an active
research field to date.\smallskip

This paper is concerned with equation \eqref{grafspde} in the simplest 
non-trivial case of a one-dimensional stochastic forcing when  $\varphi=1$,
and graphs over the unit cube in $\R^N$ with periodic boundary condition, that 
is over the flat torus $\mt$. 
This yields the SPDE initial-boundary-value problem
\begin{equation}\label{eq:smcf1}
\begin{split}
\dif u&=\sq\diver\left(\frac{\nabla u}{\sq}\right)\dif t+\sq\circ\dif W,\\
u(0)&=u_0,\qquad t\in(0,T),\, x\in\mt,
\end{split}
\end{equation}
where $\sq=\sqrt{1+|\nabla u|^2}$, $W$ is a real-valued one-dimensional  
Wiener process and
$\circ$ denotes the Stratonovich product. 

We  emphasize that this case is 
contained in the theory famously announced by Lions and Souganidis in
\cite{MR1647162,MR1659958,MR1799099,MR1807189}. In this paper, however, our aim is to
introduce an alternative approach that is based on energy methods and that 
yields the existence of weak martingale solutions to \eqref{eq:smcf1}. Even if 
we consider here a more restrictive setting, we believe that our approach can 
be 
extended to more general situations (see e.g.\ remark \ref{rem:morgen} below)
 and might be very 
helpful in problems where a comparison principle and viscosity solution 
formulations are not available. 

The use of energy methods is motivated by the gradient flow structure of the deterministic mean curvature flow. We prove that also in the case of \eqref{eq:smcf1} we retain a control over the surface area energy and over the times-space integral of the squared mean curvature, see Proposition \ref{prop:area}. In the deterministic case, in addition one often can prove an $L^\infty$ bound for the gradient (see for example \cite{EcHu91}) and consequently the uniform ellipticity of the mean curvature operator. Such a bound is typically obtained from an evolution equation derived for the function $\sqrt{1+|\nabla u|^2}$ and cannot be expected for the stochastic equation \eqref{eq:smcf1}. In contrast, our approach is based on an $L^2$ bound for $\nabla u$, see Proposition \ref{prop:est1}. These bounds are carefully exploited in a three step approximation and corresponding passages to the limit. Several refined and original tightness and identification arguments together with compensated compactness and Young measures 
techniques are required, that we believe are of independent interest.

We do not use here any more refined monotonicity properties that are often employed in the deterministic case (in particular to study singularities), most notably Huisken's monotonicity formula \cite{Huis90,Ecke04}. Such formulas are deduced from the time derivative of the surface integral over particular test functions (typically backward heat kernels). In our case the corresponding time differential comes with quite some additional terms from the It\^o--Stratonovich correction and the It\^o formula. It is not clear that appropriate cancellation properties allow to control such terms. Therefore, it is only the monotonicity property (in the corresponding deterministic equation) of the total area that we use here (or rather the control of the total area that still holds for  \eqref{eq:smcf1}). A pure PDE approach that also does not rely on any refined monotonicity formulas (but crucially on the decrease of total area) has been used in \cite{EvSp91} to prove the existence of level set solutions to mean curvature flow. We use here some ingredients of their work (in particular a compensated compactness argument) but have to deal with some additional difficulties, such as the fact that no maximum estimate for the gradient is available in our case.
%
%
\section{Mathematical framework and main results}

Our main result is the existence of weak martingale solutions to the It\^o form of \eqref{eq:smcf1} in the case $N=2$. By a direct calculation one can verify that the It\^ o-Stratonovich correction corresponding to the stochastic integral in \eqref{eq:smcf1} is
$$\frac{1}{2}\frac{\nabla u}{H(\nabla u)}\otimes\frac{\nabla u}{H(\nabla u)}:\totdif^2 u\,\dif t$$
and hence, in view of \eqref{eq:H-term}, equation \eqref{eq:smcf1} rewrites as 
\begin{equation}\label{eq:smcf2}
\begin{split}
\dif u&=\frac{1}{2}\Delta u\,\dif t+\frac{1}{2}H(\nabla u)\diver\bigg(\frac{\nabla u}{H(\nabla u)}\bigg)\,\dif t+\sq\,\dif W,\\
u(0)&=u_0,\qquad t\in(0,T),\, x\in\mathbb T^N,
\end{split}
\end{equation}
or equivalently
\begin{equation}\label{eq:smcf3}
\begin{split}
\dif u&=\Delta u\,\dif t-\frac{1}{2}\frac{\nabla u}{H(\nabla u)}\otimes\frac{\nabla u}{H(\nabla u)}:\totdif^2 u\,\dif t+\sq\,\dif W,\\
u(0)&=u_0,\qquad t\in(0,T),\, x\in\mathbb T^N.
\end{split}
\end{equation}
As we aim at establishing existence of a solution to \eqref{eq:smcf2} that is weak in both probabilistic and PDEs sense, let us introduce these two notions. From the point of view of the theory of
PDEs, we consider solutions that satisfy \eqref{eq:smcf2} in the sense of distributions and that fulfill a suitable surface area energy inequality. This implies in particular that the mean curvature belongs to $L^2$ with respect to the surface area measure $H(\nabla u)$.

From the probabilistic point of view, two concepts of solution are typically considered in the theory of stochastic evolution equations, namely, pathwise (or strong) solutions and martingale (or
weak) solutions. In the former notion the underlying probability space as well as the driving process is fixed in advance while
in the latter case these stochastic elements become part of the solution of the problem. Clearly, existence of a pathwise solution is stronger and implies existence of a martingale solution. In the present work we establish existence of a martingale solution to \eqref{eq:smcf2}. Due to the classical Yamada-Watanabe-type argument (see e.g. \cite{krylov}, \cite{pr07}), existence of a pathwise solution would then follow if pathwise uniqueness held true, however, uniqueness for \eqref{eq:smcf2} is out of the scope of the present article. In hand with this issue goes the way how the initial condition is posed: we are given a Borel probability measure on $H^1(\mt)$, hereafter denoted by $\Lambda$, that fulfills some further assumptions specified in Theorem \ref{thm:main} and plays the role of an initial law for \eqref{eq:smcf2}, that is, we require that the law of $u(0)$ coincides with $\Lambda$.

\begin{defin}\label{def:sol}
Let $\Lambda$ be a Borel probability measure on $H^1(\mt)$. Then
$$\big((\Omega,\mf,(\mf_t),\prst),u,W\big)$$
is called a weak martingale solution to \eqref{eq:smcf2} with the initial law $\Lambda$ provided
\begin{enumerate}
\item $(\Omega,\mf,(\mf_t),\prst)$ is a stochastic basis with a complete right-continuous filtration,
\item $W$ is a real-valued $(\mf_t)$-Wiener process,
\item $u\in L^2(\Omega\times[0,T],\mathcal{P},\dif\prst\otimes\dif t;H^1(\mt))$,\footnote{$\mathcal{P}$ denotes the predictable $\sigma$-algebra on $\Omega\times[0,T]$ associated to $(\mf_t)_{t\geq0}$}
\item the area measure $H(\nabla u)$ belongs to $L^1(\Omega,L^\infty(0,T;L^1(\mt)))$,
\item the mean curvature
\begin{equation*}
 v=\diver\bigg(\frac{\nabla  u}{H(\nabla u)}\bigg)
\end{equation*}
belongs to $L^2\big(\Omega\times[0,T]\times\mt,H(\nabla u)\,\dif\prst\otimes\dif t\otimes\dif x\big)$,
\item there exists a $\mf_0$-measurable random variable $u_0$ such that $\Lambda=\prst\circ u_0^{-1}$ and for every $\varphi\in C^\infty(\mt)$ it holds true for a.e. $t\in[0,T]$ a.s.
\begin{equation*}
 \begin{split}
  \langle  u(t),\varphi\rangle&=\langle  u_0,\varphi\rangle-\frac{1}{2}\int_0^t\langle  \nabla u,\nabla\varphi\rangle\dif s+\frac{1}{2}\int_0^t\langle H(\nabla u) v,\varphi\rangle\dif s+\int_0^t\langle H(\nabla u)\dif W,\varphi\rangle.
 \end{split}
\end{equation*}
\end{enumerate}
 
\end{defin}

\begin{rem}
According to Definition \ref{def:sol}(vi), equation \eqref{eq:smcf2} is satisfied in $H^{-1}(\mt)$. In particular, the solution $ u$ regarded as a class of equivalence in
$$L^2(\Omega\times[0,T],\mathcal{P},\dif\prst\otimes\dif t;H^1(\mt))$$
has a representative $\bar{u}$ with almost surely continuous trajectories in $H^{-1}(\mt)$ and moreover $\bar{u}(0)=u_0$.

\end{rem}

With this definition at hand we can formulate our main result. 

\begin{thm}\label{thm:main}
Assume $N=2$ and that the initial law $\Lambda$ satisfies\footnote{Here and in the sequel, we write $L^2_x$ for $L^2(\mt)$ and similarly for other spaces.}
\begin{equation}\label{eq:data}
\int_{H^1_x}\|\nabla z\|_{L^2_x}^2\,\dif\Lambda(z)<\infty.
\end{equation}
Then there exists a weak martingale solution to \eqref{eq:smcf2} with the initial law $\Lambda$.
\end{thm}

Our proof relies on a three step approximation scheme. The mean curvature operator $\mathcal L u = \sq\diver\left(\frac{\nabla u}{\sq}\right)$ is elliptic but not uniformly elliptic. As key approximation step we therefore  add an artificial viscosity term $\eps\Delta u$. This viscous approximation of the target equation \eqref{eq:smcf1} might be interesting also in its own right. Existence is still not immediate as the equation is quasilinear and the nonlinearities in the equation are not Lipschitz. We therefore add two further regularizations and firstly increase the order of the equation by adding a term $-\eta\Delta^{2K} u$, $K\in\N$ sufficiently large, which in turn yields a semilinear (nondegenerate) parabolic SPDE. We secondly replace $\totdif^2 u$ in \eqref{eq:H-term} by a suitable uniformly bounded truncation $\Theta^R(\totdif^2 u)$, such that the corresponding nonlinearity is Lipschitz. For the resulting equation existence of a unique mild solution is deduced by semigroup arguments.

To obtain the existence of the original equation we pass to the limit with the respective regularizations. Convergence with $R\to\infty$ can be performed via a stopping time argument (Theorem \ref{thm:regularized}). Using the stochastic compactness method,  in Section \ref{sec:viscous} we let $\eta\to 0$.  In particular, in Theorem \ref{prop:martsol} we obtain existence of a strong martingale solution to the viscous approximation of \eqref{eq:smcf1} in any space dimension. 

The most challenging part is the passage to the limit $\eps\to 0$ in Section \ref{sec:van-vis}. In Proposition \ref{prop:area} we first derive a crucial uniform estimate for the surface area and mean curvature. It is this step where we need to restrict ourselves to spatial dimension $N=2$; at some point we need a cancellation property of terms that originate from the Stratonovich-It\^o correction and the It\^o chain rule, respectively. This property uses a Gauss--Bonnet type formula that is only valid for surfaces (here we also exploit the periodic boundary condition). Still, to pass to the limit $\eps\to 0$ we have to overcome several substantial difficulties. In particular, the only available estimate for higher order derivatives is given by the mean curvature bound. However, both the $L^2$ gradient bound and the bound on area and mean curvature are not available for higher moments. 

Therefore, in order to identify the limit of the nonlinear terms in the equation, we proceed in several steps. First, in Proposition \ref{prop:skorokhod}, it is not enough to prove tightness for the approximate solutions only so we also include some (nonlinear) functionals of their gradients. This leads us to the Jakubowski-Skorokhod representation theorem (see \cite{jakubow}), which is valid in a large class of topological spaces that are not necessarily metrizable but retain several important properties of Polish spaces. However, the implied convergence still does not suffice hence in Proposition \ref{prop:young} we employ compensated compactness and Young measure arguments to deduce a crucial strong convergence property of the gradients in suitable $L^p$-spaces. Note that a similar method was already used in the context of mean curvature flow in \cite{EvSp95} (although rather for the convergence towards the level set formulation).

Finally, in Subsection \ref{subsec:5.3}, we employ a refined identification procedure for the stochastic integral. It is based on a general method of constructing martingale solutions in the absence of suitable martingale representation theorems introduced in \cite{on1} as well as a method of densely defined martingales from \cite{hof} and a local martingale approach from \cite{hofse}.
More precisely, the method of densely defined martingales which was developed in \cite{hof} is applied in order to deal with martingales that are only defined for almost all times and no continuity properties are a priori known (see \cite[Theorem 4.13, Appendix]{hof}). In that case, the corresponding quadratic variations are not well defined and the approach of Subsection \ref{subsec:identif} does not apply directly.
Second, the local martingales approach of \cite{hofse} is invoked to overcome the difficulty in the passage to the limit.

Both issues originate in the lack of uniform moment estimates for $\nabla u$. Indeed, on the one hand, we are not able to obtain tightness of the approximate solutions in any space of continuous (or weakly continuous) functions in time and consequently the passage to the limit in the corresponding martingales can be performed only for a.e. $t$.
On the other hand, we are only able to establish the strong convergence of the gradients in $L^p_{\omega,t,x}$ for $p\in[1,2)$ and the convergence in $L^2_{\omega,t,x}$ remains weak, which is not enough to pass to the limit in the quadratic variation.

Let us also mention that, since we are dealing with solutions that are weak in the PDE sense, the classical infinite dimensional It\^o formula (see for instance \cite{daprato}) does not apply and hence, in order to establish the a priori estimates rigorously, one is led to a suitable generalized version. Several times throughout the paper we therefore refer to \cite[Proposition A.1]{dehovo} which provides such a generalization for a wide class of quasilinear parabolic SPDEs. In the same spirit one can justify the computations in the present paper.

\begin{rem} \label{rem:morgen} Theorem \ref{thm:main} remains true in the 
slightly generalized 
case of the SPDE
\[ \dif u=\sq\diver\left(\frac{\nabla u}{\sq}\right)\dif t+ \sqrt  
\rho\,  \sq\circ\dif W,\]
provided the parameter $\rho >0 $ modeling the strength of the 
noise is not too large, i.e.\ $\rho <2$. This is in perfect agreement  with 
the corresponding well posedness result in 1D obtained by completely different 
means in \cite[Theorem 3.3]{MR2888287}. Indeed, in the It\^o representation 
the SPDE above reads
\[ \dif u=  \frac \rho 2 \Delta u \dif t + (1- \frac \rho 2 ) \mathcal L 
u \, \dif t  + \sqrt  
\rho\,  \sq \dif W,\]
where $\mathcal L u = \sq\diver\left(\frac{\nabla u}{\sq}\right)$ is the mean 
curvature operator. Note that for $\rho>0$ the drift is uniformly elliptic. 
Moreover, in  the crucial gradient energy 
estimate the first term in the drift exactly dissipates the influx of the
energy through  the noise term. In the simplified case when $\mathcal 
L$ is replaced by zero this can immediately  be seen from the corresponding It\^o
formula. (Full details for our case involving 
$\mathcal L$ are given in  Proposition \ref{prop:est1} below.)  
In particular, the gradient $\nabla u$ is controlled in $L^2$ such that the solution remains a graph.

\smallskip It would be interesting to consider the case of a more general noise term of the 
form 
\[\sum_{i=1}^K \eta_i(x) H(\nabla u) \circ \dif W^{i}\] for independent 
Wiener processes $\{W^{i}, i=1, \cdots, K\}$, under appropriate boundedness conditions 
on the smooth coefficient functions  $\{\eta_i, i=1, \cdots ,K\}$. If one proceeds by similar means 
as for \eqref{eq:smcf1} even in the case $K=1$ computations quickly become very involved, and it is not clear whether similar cancellation properties hold in this case.

\end{rem}

\section{Regularized equation}
\label{sec:regul}

To begin with, let $(\Omega,\mf,(\mf_t)_{t\geq0},\prst)$ be a stochastic basis with a complete, right-continuous filtration and let $W$ be a real-valued Wiener process relative to $(\mf_t)$. In order to prepare the initial data for the first approximation layer, let $u_0$ be a $(\mf_0)$-measurable random variable with the law $\Lambda$ and for $\varepsilon\in(0,1)$ let $u_0^{\varepsilon}$ be an approximation of $u_0$ such that for all $p\in[2,4)$
$$\stred\|u^{\varepsilon}_0\|^p_{H^k_x}\leq C_{\varepsilon}$$
where we fixed $k\in\mn$ such that $k>2+N/2$.
Let the law of $u_0^{\varepsilon}$ on $H^k(\mt)$ be denoted by $\Lambda^{\varepsilon}$. Then according to \eqref{eq:data} the following estimate holds true uniformly in $\varepsilon$
\begin{equation}\label{eq:dataappr}
\int_{H^1_x}\|\nabla z\|_{L^2_x}^2\,\dif\Lambda^{\varepsilon}(z)=\stred\|\nabla u^{\varepsilon}_0\|_{L^2_x}^2\leq C
\end{equation}
and $\Lambda^\varepsilon\overset{*}{\rightharpoonup} \Lambda$ in the sense of measures on $H^1(\mt)$.

As the first step in the proof of existence for \eqref{eq:smcf2}, we consider its equivalent form \eqref{eq:smcf3} and approximate in the following way\footnote{Here and in the sequel, $A^*$ denotes the transpose of a matrix $A$.}
\begin{equation}\label{eq:approx}
\begin{split}
\dif u&=(1+\varepsilon)\Delta u\,\dif t-\frac{1}{2}\frac{(\nabla u)^*}{\sq}\totdif^2 u\frac{\nabla u}{\sq}\,\dif t-\eta\Delta^{2K}u\,\dif t+\sq\,\dif W,\\
u(0)&=u^{\varepsilon}_0.
\end{split}
\end{equation}
Our aim here is to establish an existence result for $\varepsilon,\eta$ fixed and $K$ sufficiently large.

\begin{thm}\label{thm:regularized}
Let $k,K\in\mn$ be such that $2+N/2<k<2K$ and let $p>2$. Assume that $u^\varepsilon_0\in L^p(\Omega;H^k)$. Then for any $\varepsilon,\eta\in(0,1)$ there exists $u\in L^p(\Omega\times[0,T],\mathcal{P},\dif\prst\otimes\dif t;H^k)$ that is the unique mild solution to \eqref{eq:approx}.

\begin{proof}
In order to guarantee the Lipschitz property of the nonlinear second order term in \eqref{eq:approx}, let $R\in\mn$ and consider the truncated problem
\begin{equation}\label{eq:approxtr}
\begin{split}
\dif u&=(1+\varepsilon)\Delta u\,\dif t-\frac{1}{2}\frac{(\nabla u)^*}{\sq}\Theta^R(\totdif^2 u)\frac{\nabla u}{\sq}\,\dif t-\eta\Delta^{2K}u\,\dif t+\sq\,\dif W,\\
u(0)&=u^{\varepsilon}_0,
\end{split}
\end{equation}
where $\Theta^R:\mr^{N\times N}\rightarrow \mr^{N\times N}$ is a truncation, i.e. for $A=(a_{ij})\in\mr^{N\times N}$ we define $\Theta^R(A)=\big(\theta^R(a_{ij})a_{ij}\big)$ where $\theta^R:\mr\rightarrow[0,1]$ is a smooth truncation satisfying
$$\theta^R(\xi)=\begin{cases}
                 1,&|\xi|\leq R/2\\
		 0,&|\xi|\geq R.
                \end{cases}
$$

Let $S$ denote the semigroup generated by the strongly elliptic differential operator $\eta\Delta^{2K}-(1+\varepsilon)\Delta$. Let $\mathcal{H}=L^p(\Omega\times[0,T],\mathcal{P},\dif\prst\otimes\dif t;H^k)$ and define the mapping
\begin{equation*}
 \begin{split}
\mathcal K u(t)&=S(t)u^{\varepsilon}_0-\frac{1}{2}\int_0^t S(t-s)\frac{(\nabla u)^*}{H(\nabla u)}\Theta^R(\totdif^2 u)\frac{\nabla u}{H(\nabla u)}\,\dif s\\
&\qquad+\int_0^t S(t-s) H(\nabla u)\,\dif W.  
 \end{split}
\end{equation*}
Then $\mathcal{K}$ maps $\mathcal{H}$ into $\mathcal{H}$ and it is a contraction. Indeed, using the regularization property of the semigroup and Young's inequality for convolutions we obtain for any $u\in\mathcal{H}$ (provided $k<2K$)
\begin{equation}\label{eq:growth}
\begin{split}
\bigg\|\int_0^\tec S&(\cdot-s)\frac{(\nabla u)^*}{H(\nabla u)}\Theta^R(\totdif^2 u)\frac{\nabla u}{H(\nabla u)}\,\dif s\bigg\|_{\mathcal{H}}^p\\
&\leq \stred\int_0^T\bigg(\int_0^t\bigg\|S(t-s)\frac{(\nabla u)^*}{H(\nabla u)}\Theta^R(\totdif^2 u)\frac{\nabla u}{H(\nabla u)}\bigg\|_{H^k}\dif s\bigg)^p\dif t\\
&\leq C\,\stred\int_0^T\bigg(\int_0^t(t-s)^{-k/4K}\bigg\|\frac{(\nabla u)^*}{H(\nabla u)}\Theta^R(\totdif^2 u)\frac{\nabla u}{H(\nabla u)}\bigg\|_{L^2}\dif s\bigg)^p\dif t\\
&\leq C T^{p(1-k/4K)}\|u\|_{\mathcal{H}}^p,
\end{split}
\end{equation}
and similarly for the stochastic term where we apply the Burkholder-Davis-Gundy inequality first (see e.g. \cite{b1})
\begin{equation*}
\begin{split}
\bigg\|\int_0^\tec S(\cdot-s)&H(\nabla u)\,\dif W\bigg\|_{\mathcal{H}}^p\leq C\int_0^T\stred\bigg(\int_0^t\big\|S(t-s)H(\nabla u)\big\|_{H^k}^2\dif s\bigg)^{p/2}\dif t\\
&\leq C\,\stred\int_0^T\bigg(\int_0^t(t-s)^{-k/2K}\big\|H(\nabla u)\big\|_{L^2}^2\dif s\bigg)^{p/2}\dif t\\
&\leq CT^{p/2(1-k/2K)}\|u\|_{\mathcal{H}}^p.
\end{split}
\end{equation*}
In order to verify the contraction property, we observe that for any $u,v\in\mathcal{H}$
\begin{equation*}
\begin{split}
\bigg\|\frac{(\nabla u)^*}{H(\nabla u)}\Theta^R(\totdif^2 u)&\frac{\nabla u}{H(\nabla u)}-\frac{(\nabla v)^*}{H(\nabla v)}\Theta^R(\totdif^2 v)\frac{\nabla v}{H(\nabla v)}\bigg\|_{L^2}\\
&\leq \big\|\totdif^2 u-\totdif^2 v\big\|_{L^2}+C_R\big\|\nabla u-\nabla v\big\|_{L^2}\leq C_R\|u-v\|_{H^k}
\end{split}
\end{equation*}
hence by a similar approach as above
\begin{equation*}
\begin{split}
\bigg\|\int_0^\tec S&(\cdot-s)\bigg(\frac{(\nabla u)^*}{H(\nabla u)}\Theta^R(\totdif^2 u)\frac{\nabla u}{H(\nabla u)}-\frac{(\nabla v)^*}{H(\nabla v)}\Theta^R(\totdif^2 v)\frac{\nabla v}{H(\nabla v)}\bigg)\,\dif s\bigg\|_{\mathcal{H}}^p\\
&\leq C_R\,\stred\int_0^T\bigg(\int_0^t(t-s)^{-k/4K}\| u- v\|_{H^k}\dif s\bigg)^p\dif t\leq C_R\,T^{p(1-k/4K)}\|u-v\|_\mathcal{H}^p
\end{split}
\end{equation*}
and
\begin{equation*}
\begin{split}
\bigg\|\int_0^\tec S(\cdot-s)&\big(H(\nabla u)-H(\nabla v)\big)\,\dif W\bigg\|_{\mathcal{H}}^p\\
&\leq C\,\stred\int_0^T\bigg(\int_0^t(t-s)^{-k/2K}\big\|H(\nabla u)-H(\nabla v)\big\|_{L^2}^2\dif s\bigg)^{p/2}\dif t\\
&\leq CT^{p/2(1-k/2K)}\|u-v\|_{\mathcal{H}}^p.
\end{split}
\end{equation*}
Therefore, if $T$ is small enough then $\mathcal{K}$ has unique fixed point $u$ in $\mathcal{H}$ which is the mild solution of \eqref{eq:approxtr}. Furthermore, by a standard use of the factorization lemma (see \cite{daprato}), it has continuous trajectories with values in $H^k$ provided $K$ is large enough, i.e. it belongs to $L^p(\Omega;C([0,T];H^k)).$
Therefore, the condition on $T$ can be easily removed by considering the equation on smaller intervals $[0,\tilde{T}],\,[\tilde{T},2\tilde{T}],$ etc.

As a consequence, for any $R\in\mn$ there exists a unique mild solution to \eqref{eq:approxtr}, let it be denoted by $u^R$. Furthermore, since $k>\frac{N}{2}+2$, it follows from the Sobolev embedding theorem that
\begin{equation}\label{eq:blowup}
\stred\sup_{0\leq t\leq T}\big\|\totdif^2 u^R\big\|_{L^\infty}\leq C,
\end{equation}
where the constant on the right hand side is independent of $R$ (this can be seen from the fact that the growth estimates for the nonlinear second order term in \eqref{eq:approxtr} do not depend on $R$, cf. \eqref{eq:growth}). Hence
$$\tau_R=\inf\Big\{t>0;\big\|\totdif^2 u^R\big\|_{L^\infty}\geq R/2\Big\}$$
(with the convention $\inf\emptyset=T$) defines an $(\mf_t)$-stopping time and $u^R$ is a solution to \eqref{eq:approx} on $[0,\tau_R)$. Besides, due to uniqueness, if $R'>R$ then $\tau_{R'}\geq\tau_R$ and $u^{R'}=u^{R}$ on $[0,\tau_R)$. Moreover, the blow up cannot occur in a finite time by \eqref{eq:blowup} so
$$\tau=\sup_{R\in\mn}\tau_R=T\qquad\text{a.s.}$$
and therefore the process $u$ which is uniquely defined by $u:=u^R$ on $[0,\tau_R)$ is the unique mild solution to \eqref{eq:approx} on $[0,T]$.
\end{proof}
\end{thm}

\section{Viscous approximation}
\label{sec:viscous}

Having Theorem \ref{thm:regularized} in hand it is necessary to find sufficient estimates uniform in $\eta$ in order to justify the passage to the limit as $\eta\rightarrow0$ and obtain a martingale solution to
\begin{equation}\label{eq:approx1}
\begin{split}
\dif u&=(1+\varepsilon)\Delta u\,\dif t-\frac{1}{2}\frac{(\nabla u)^*}{\sq}\totdif^2 u\frac{\nabla u}{\sq}\,\dif t+\sq\,\dif W
\end{split}
\end{equation}
with the initial law $\Lambda$.
Let us fix $\eps>0$ and denote by $u^\eta$ the solution to \eqref{eq:approx} given by Theorem \ref{thm:regularized}. Recall that $\prst\circ u^\eta(0)^{-1}=\Lambda^{\varepsilon}$ for all $\eta\in(0,1)$ and the uniform estimate \eqref{eq:dataappr} holds true.

\begin{prop}\label{prop:est1}
For any $p\in[2,2(1+\varepsilon)]$ it holds true
\begin{equation}\label{eq:est1}
\begin{split}
\stred\|\nabla u^\eta(t)\|_{L^2}^p+\frac{p(2(1+\varepsilon)-p)}{2}\stred\int_0^t\|\nabla u^\eta\|_{L^2}^{p-2}\|\Delta u^\eta\|_{L^2}^2\leq\stred\|\nabla u_0^{\varepsilon}\|_{L^2}^p\leq C,
\end{split}
\end{equation}
where the constant $C$ is independent of $\eta$ and if $p=2$ then it is also independent of $\varepsilon.$

\begin{proof}
Since mild solution is a weak solution, we consider the function $f(\nabla v)=\|\nabla v\|_{L^2}^p$ and apply similar arguments as in the generalized It\^o formula \cite[Proposition A.1]{dehovo} to obtain
\begin{equation*}
\begin{split}
\stred\|\nabla u^\eta(t)\|_{L^2}^p&=\stred\|\nabla u^{\varepsilon}_0\|_{L^2}^p+p(1+\varepsilon)\stred\int_0^t\|\nabla u^\eta\|_{L^2}^{p-2}\big\langle\nabla u^\eta,\nabla\Delta u^\eta\big\rangle\,\dif s\\
&\quad-p\eta\stred\int_0^t\|\nabla u^\eta\|_{L^2}^{p-2}\big\langle\nabla u^\eta,\nabla\Delta^{2K} u^\eta\big\rangle\,\dif s\\
&\quad-\frac{p}{2}\stred\int_0^t\|\nabla u^\eta\|_{L^2}^{p-2}\bigg\langle\nabla u^\eta,\nabla\bigg(\frac{(\nabla u^\eta)^*}{H(\nabla u^\eta)}\totdif^2 u^\eta\frac{\nabla u^\eta}{H(\nabla u^\eta)}\bigg)\bigg\rangle\,\dif s\\
&\quad+\frac{p}{2}\stred\int_0^t\|\nabla u^\eta\|_{L^2}^{p-2}\big\|\nabla H(\nabla u^\eta)\big\|_{L^2}^2\dif s\\
&\quad+ \frac{p(p-2)}{2}\stred\int_0^t\|\nabla u^\eta\|_{L^2}^{p-4}\big\langle\nabla u^\eta,\nabla H(\nabla u^\eta)\big\rangle^2\dif s\\
&=J_1+\cdots+J_6.
\end{split}
\end{equation*}
It holds
\begin{equation*}
\begin{split}
J_2+J_3&\leq -p(1+\varepsilon)\stred\int_0^t\|\nabla u^\eta\|_{L^2}^{p-2}\|\Delta u^\eta\|_{L^2}^2\dif s,\\
J_4&\leq\frac{p}{2}\stred\int_0^t\|\nabla u^\eta\|_{L^2}^{p-2}\|\Delta u^\eta\|_{L^2}^2\dif s,\\
J_5+J_6&\leq \frac{p(p-1)}{2}\stred\int_0^t\|\nabla u^\eta\|_{L^2}^{p-2}\|\Delta u^\eta\|_{L^2}^2\dif s,
\end{split}
\end{equation*}
where we used the fact that $\|\Delta u^\eta\|_{L^2}=\|\totdif^2 u^\eta\|_{L^2}$ due to boundary conditions. Hence the claim follows.
\end{proof}
\end{prop}

\begin{prop}\label{prop:L2}
It holds true
\begin{equation*}\label{eq:L2a}
\begin{split}
\stred\|u^\eta(t)\|_{L^2}^2\leq C_\varepsilon\big(1+\stred\| u^{\varepsilon}_0\|_{L^2}^2\big)\leq C,
\end{split}
\end{equation*}
where the constant $C$ is independent of $\eta$.

\begin{proof}
With regard to Proposition \ref{prop:est1}, the above estimate is a consequence of the It\^o formula applied to the function $f(v)=\| v\|_{L^2}^2$:
\begin{equation*}
\begin{split}
\stred\| u^\eta(t)\|_{L^2}^2&=\stred\| u^{\varepsilon}_0\|_{L^2}^2+2(1+\varepsilon)\stred\int_0^t\big\langle u^\eta,\Delta u^\eta\big\rangle\,\dif s-2\eta\stred\int_0^t\big\langle u^\eta,\Delta^{2K} u^\eta\big\rangle\,\dif s\\
&\quad-\stred\int_0^t\bigg\langle u^\eta,\frac{(\nabla u^\eta)^*}{H(\nabla u^\eta)}\totdif^2 u^\eta\frac{\nabla u^\eta}{H(\nabla u^\eta)}\bigg\rangle\,\dif s+\stred\int_0^t\big\| H(\nabla u^\eta)\big\|_{L^2}^2\dif s\\
&=J_1+\cdots+J_5.
\end{split}
\end{equation*}
Similar arguments as above imply
\begin{equation*}
\begin{split}
J_2+J_3&\leq -2(1+\varepsilon)\stred\int_0^t\|\nabla u^\eta\|_{L^2}^2\dif s,\\
J_4&\leq\frac{1}{2}\stred\int_0^t\| u^\eta\|_{L^2}^2\dif s+\frac{1}{2}\stred\int_0^t\|\Delta u^\eta\|_{L^2}^2,\\
J_5&\leq \stred\int_0^t\|\nabla u^\eta\|_{L^2}^2\dif s,
\end{split}
\end{equation*}
hence Proposition \ref{prop:est1} and the Gronwall lemma completes the proof.
\end{proof}
\end{prop}

\begin{prop}\label{prop:holder}
Let $p\in(2,2(1+\varepsilon))$. Then for any $\alpha\in(1/p,1/2)$ there exists $C_\varepsilon>0$ such that
\begin{equation}\label{eq:holder}
\stred\|u^\eta\|_{C^{\alpha-1/p}([0,T];H^{2-4K})}^p\leq C_\varepsilon.
\end{equation}

\begin{proof}
According to Propositions \ref{prop:est1} and \ref{prop:L2}, $u^\eta\in L^2(\Omega;L^2(0,T;H^2(\mt)))$ uniformly in $\eta$ (not $\varepsilon$). As a consequence,
$$(1+\varepsilon)\Delta u^\eta-\frac{1}{2}\frac{(\nabla u^\eta)^*}{H(\nabla u^\eta)}\totdif^2 u^\eta\frac{\nabla u^\eta}{H(\nabla u^\eta)}-\eta\Delta^{2K}u^\eta$$
belongs to $L^2(\Omega;L^2(0,T;H^{2-4K}))$ uniformly in $\eta$ and
\begin{equation*}
 \stred\Big\|u^\eta-\int_0^\tec H(\nabla u^\eta)\dif W\Big\|_{C^{1/2}([0,T];H^{2-4K})}\leq C_\varepsilon.
\end{equation*}
Concerning the stochastic integral, we have by factorization and \eqref{eq:est1}
\begin{equation*}
\begin{split}
 \stred \bigg\|\int_0^\tec H(\nabla u^\eta)\dif W&\bigg\|_{C^{\alpha-1/p}([0,T];L^2)}^p\leq C\,\stred\bigg\|\int_0^\tec(\cdot-s)^{-\alpha} H(\nabla u^\eta)\dif W(s)\bigg\|_{L^p(0,T;L^2)}^p\\
&\leq C\int_0^T\stred\bigg(\int_0^t(t-s)^{-2\alpha}\big(1+\|\nabla u^\eta\|_{L^2}^2\big)\dif s\bigg)^{p/2}\dif t\\
&\leq C\, T^{p/2(1-2\alpha)}\stred\int_0^T\big(1+\|\nabla u^\eta\|_{L^2}^p\big)\dif t\leq C
\end{split}
\end{equation*}
and the claim follows.
\end{proof}
\end{prop}

Now we would like to pass to the limit $\eta\searrow 0$.

\subsection{Compactness}
\label{subsec:compact}

Let us define the path space $\mathcal{X}=\mathcal{X}_u\times\mathcal{X}_v\times\mathcal{X}_W$, where\footnote{If a topological space $X$ is equipped with the weak topology we write $(X,w)$.}
$$\mathcal{X}_u=L^2\big(0,T;H^1(\mt)\big)\cap C\big([0,T];H^{1-4K}(\mt)\big),\quad\mathcal{X}_v=\big((L^2(0,T;L^2(\mt)),w\big),$$
$$\mathcal{X}_W=C\big([0,T];\mr\big).$$
Let us denote by $\mu_{u^\eta}$ the law of $u^\eta$ on $\mathcal{X}_u$, $\eta\in(0,1)$, by $\mu_{v^\eta}$ the law of
$$v^\eta:=\diver\bigg(\frac{\nabla u^\eta}{H(\nabla u^\eta)}\bigg)$$
on $\mathcal{X}_v$ and by $\mu_W$ the law of $W$ on $\mathcal{X}_W$. Their joint law on $\mathcal{X}$ is then denoted by $\mu^\eta$.

\begin{prop}\label{tight}
The set $\{\mu^\eta;\,\eta\in(0,1)\}$ is tight on $\mathcal{X}.$

\begin{proof}
First, we prove tightness of $\{\mu_{u^\eta};\,\eta\in(0,1)\}$ which follows directly from Proposition \ref{prop:est1} and \ref{prop:holder} by making use of the embeddings
\begin{equation*}
\begin{split}
C^{\alpha-1/p}([0,T];H^{2-4K}(\mt))&\hookrightarrow H^\lambda(0,T;H^{2-4K}(\mt)),\;\lambda<\alpha-1/p,\\
C^{\alpha-1/p}([0,T];H^{2-4K}(\mt))&\overset{c}{\hookrightarrow} C([0,T];H^{1-4K}(\mt)),\\
L^2(0,T;H^2(\mt))\cap H^{\lambda}(0,T;H^{2-4K}(\mt))&\overset{c}{\hookrightarrow} L^2(0,T;H^1(\mt)),\;\lambda>0.\\
\end{split}
\end{equation*}
Here the first embedding follows immediately for the definition of the spaces, the second one is a consequence of the Arzel\`a-Ascoli theorem and the third one can be found in \cite{fland}.
For $R>0$ let us define the set
\begin{multline*}
B_{R}=\big\{u\in L^2(0,T;H^{2}(\mt))\cap C^{\alpha-1/p}([0,T];H^{2-4K}(\mt));\\
\|u\|_{L^2(0,T;H^{2})}+\|u\|_{C^{\alpha-1/p}([0,T];H^{2-4K})}\leq R\big\}
\end{multline*}
which is thus relatively compact in $\mathcal{X}_u$. Moreover, by Propositions \ref{prop:est1}, \ref{prop:L2} and \ref{prop:holder}
\begin{equation*}
\begin{split}
\mu_{u^\eta}\big(B_{R}^C\big)&\leq\prst\bigg(\|u^\eta\|_{L^2(0,T;H^{2})}>\frac{R}{2}\bigg)+\prst\bigg(\|u^\eta\|_{C^{\alpha-1/p}([0,T];H^{2-4K})}>\frac{R}{2}\bigg)\\
&\leq\frac{C}{R^2}\stred\|u^\eta\|^2_{L^2(0,T;H^{2})}+\frac{C}{R^p}\stred\|u^\eta\|^p_{C^{\alpha-1/p}([0,T];H^{2-4K})}\leq\frac{C}{R^2}
\end{split}
\end{equation*}
hence given $\vartheta>0$ there exists $R>0$ such that
$$\mu_{u^\eta}(B_{R})\geq 1-\vartheta$$
which yields the claim.

Concerning the tightness of $\{\mu_{v^\eta};\,\eta\in(0,1)\}$ we proceed similarly and make use of the uniform estimate from Proposition \ref{prop:est1} together with the fact that for $R>0$ the set
$$B_R=\big\{v\in L^2(0,T;L^2(\mt));\;\|v\|_{L^2(0,T;L^2)}\leq R\big\}$$
is relatively compact in $\mathcal{X}_v$.

Since the law $\mu_W$ is tight as being a Radon measure on the Polish space $\mathcal{X}_W$, we conclude that also the set of the joint laws $\{\mu^\eta\semicol\,\eta\in(0,1)\}$ is tight and the proof is complete.
\end{proof}
\end{prop}

The path space $\mathcal{X}$ is not a Polish space and so our compactness argument is based on the
Jakubowski-Skorokhod representation theorem instead of the classical Skorokhod representation
theorem, see \cite{jakubow}. To be more precise, passing to a weakly convergent subsequence $\mu^n=\mu^{\eta_n}$ (and denoting by $\mu$ the limit law) we infer the following result.

\begin{prop}\label{prop:skor}
There exists a probability space $(\tilde{\Omega},\tilde{\mf},\tilde{\prst})$ with a sequence of $\mathcal{X}$-valued random variables $(\tilde{u}^n,\tilde v^n,\tilde{W}^n),\,n\in\mn,$ and $(\tilde{u},\tilde v, \tilde{W})$ such that
\begin{enumerate}
 \item the laws of $(\tilde{u}^n,\tilde v^n,\tilde{W}^n)$ and $(\tilde{u},\tilde v,\tilde{W})$ under $\,\tilde{\prst}$ coincide with $\mu^n$ and $\mu$, respectively,
 \item $(\tilde{u}^n,\tilde v^n,\tilde{W}^n)$ converges $\,\tilde{\prst}$-almost surely to $(\tilde{u},\tilde v,\tilde{W})$ in the topology of $\mathcal{X}$.
\end{enumerate}

\end{prop}

We are immediately able to identify $\tilde v^n$, $n\in\mn,$ and $\tilde v$.

\begin{cor}\label{cor:id1}
It holds true that
\begin{align*}
\tilde v^n&=\diver\bigg(\frac{\nabla \tilde u^n}{H(\nabla \tilde u^n)}\bigg)\quad\tilde\prst\text{-a.s.}\quad \forall n\in\mn,\\
\tilde v&= \diver\bigg(\frac{\nabla \tilde u}{H(\nabla \tilde u)}\bigg)\quad\tilde\prst\text{-a.s.}
\end{align*}

\begin{proof}
According to Proposition \ref{prop:est1}, the mapping
$$\supp \mu_{u^n}\rightarrow L^2(0,T;L^2(\mt)),\;u\mapsto \diver\bigg(\frac{\nabla u}{H(\nabla u)}\bigg)$$
is well-defined and measurable and hence, the first part of the statement follows directly from the equality of joint laws of $(u^n,v^n)$ and $(\tilde u^n,\tilde v^n)$. Identification of the limit $\tilde v$ follows easily using integration by parts together with the strong convergence of $\nabla \tilde u^n$ given by Proposition \ref{prop:skor}.
\end{proof}
\end{cor}

Finally, let $(\tilde{\mf}_t)$ be the $\tilde{\prst}$-augmented canonical filtration of the process $(\tilde{u},\tilde{W})$, that is
$$\tilde{\mf}_t=\sigma\big(\sigma\big(\varrho_t\tilde{u},\varrho_t\tilde{W}\big)\cup\big\{N\in\tilde{\mf};\;\tilde{\prst}(N)=0\big\}\big),\quad t\in[0,T],$$
where $\varrho_t$ is the operator of restriction to the interval $[0,t]$ acting on various path spaces. In particular, if $X$ stands for one of the path spaces $\mathcal{X}_u$ or $\mathcal{X}_W$ and $t\in[0,T]$, we define
\begin{equation*}\label{restr}
\begin{split}
\varrho_t:X\rightarrow X|_{[0,t]},\quad f&\mapsto f|_{[0,t]}.
\end{split}
\end{equation*}
Note that $\tilde v$ is also adapted with respect to $(\tilde \mf_t)$ due to Corollary \ref{cor:id1}.

\subsection{Identification of the limit}
\label{subsec:identif}

The aim of this subsection is to establish existence of a weak martingale solution to \eqref{eq:approx1}. Similarly to \eqref{eq:smcf2} and \eqref{eq:smcf3}, we rewrite \eqref{eq:approx1} into a more convenient form given by
\begin{equation}\label{eq:approx2}
\dif u = \Big(\frac{1}{2}+\varepsilon\Big)\Delta u\,\dif t+\frac{1}{2}H(\nabla u)\diver\bigg(\frac{\nabla u}{H(\nabla u)}\bigg)\,\dif t
+H(\nabla u)\,\dif W
\end{equation}
and
\begin{equation}\label{eq:approx33}
\begin{split}
\dif  u &= \varepsilon\Delta u\,\dif t+H(\nabla u)\diver\bigg(\frac{\nabla u}{H(\nabla u)}\bigg)\dif t+\frac{1}{2}\frac{(\nabla u)^*}{H(\nabla u)}\totdif^2u\frac{\nabla u}{H(\nabla u)}\dif t+H(\nabla u)\dif W.
\end{split}
\end{equation}

\begin{thm}\label{prop:martsol}
$\big((\tilde{\Omega},\tilde{\mf},(\tilde{\mf}_t),\tilde{\prst}),\tilde{W},\tilde{u}\big)$
is a strong martingale solution to \eqref{eq:approx2} with the initial law $\Lambda^\varepsilon$, i.e. $\tilde\prst\circ \tilde u(0)^{-1}=\Lambda^\varepsilon$ and
\begin{equation}\label{eq:strong}
\tilde u(t)=\tilde u(0)+\Big(\frac{1}{2}+\varepsilon\Big)\int_0^t\Delta \tilde u\,\dif s+\frac{1}{2}\int_0^t H(\nabla\tilde u)\tilde v\,\dif s+\int_0^t H(\nabla\tilde u)\,\dif \tilde W,
\end{equation}
with
\begin{equation}\label{eq:v}
\tilde v=\diver\bigg(\frac{\nabla \tilde u}{H(\nabla \tilde u)}\bigg)
\end{equation}
holds true for all $t\in[0,T]$, almost everywhere in $(\omega,x)\in\tilde\Omega\times\mt$.

\end{thm}

The proof is based on a new general method of constructing martingale solutions of SPDEs, that does not rely on any kind of martingale representation theorem and therefore holds independent interest especially in situations where these representation theorems are no longer available.

First, we show that $\big((\tilde{\Omega},\tilde{\mf},(\tilde{\mf}_t),\tilde{\prst}),\tilde{W},\tilde{u}\big)$
is a weak martingale solution to \eqref{eq:approx2} with the initial law $\Lambda^\varepsilon$, i.e. for every $\varphi\in C^\infty(\mt)$
\begin{equation}\label{eq:weak}
 \begin{split}
  \langle \tilde u(t),\varphi\rangle&=\langle \tilde u(0),\varphi\rangle+\Big(\frac{1}{2}+\varepsilon\Big)\int_0^t\langle \tilde u,\Delta\varphi\rangle\dif s+\frac{1}{2}\int_0^t\langle \sqt\tilde v,\varphi\rangle\dif s\\
&\qquad+\int_0^t\langle \sqt\dif\tilde W,\varphi\rangle,
 \end{split}
\end{equation}
where $\tilde v$ was defined in \eqref{eq:v}.

Towards this end, let us define for all $t\in[0,T]$ and a test function $\varphi\in C^\infty(\mt)$
\begin{equation*}
\begin{split}
M^n(t)&=\big\langle u^n(t),\varphi\big\rangle-\big\langle u^n_0,\varphi\big\rangle-\Big(\frac{1}{2}+\varepsilon\Big)\int_0^t\langle u^n,\Delta\varphi\rangle\dif s\\
&\qquad-\frac{1}{2}\int_0^t\langle H(\nabla u^n) v^n,\varphi\rangle\dif s+\eta_n\int_0^t\big\langle u^n,\Delta^{2K}\varphi\big\rangle\,\dif s,\quad n\in\mn,\\
\tilde{M}^n(t)&=\big\langle \tilde u^n(t),\varphi\big\rangle-\big\langle \tilde u^n(0),\varphi\big\rangle-\Big(\frac{1}{2}+\varepsilon\Big)\int_0^t\langle \tilde u^n,\Delta\varphi\rangle\dif s\\
&\qquad-\frac{1}{2}\int_0^t\langle H(\nabla \tilde u^n)\tilde v^n,\varphi\rangle\dif s+\eta_n\int_0^t\big\langle \tilde u^n,\Delta^{2K}\varphi\big\rangle\,\dif s,\quad n\in\mn,\\
\tilde{M}(t)&=\big\langle \tilde u(t),\varphi\big\rangle-\big\langle \tilde u(0),\varphi\big\rangle-\Big(\frac{1}{2}+\varepsilon\Big)\int_0^t\langle \tilde u,\Delta\varphi\rangle\dif s-\frac{1}{2}\int_0^t\langle \sqt\tilde v,\varphi\rangle\dif s,
\end{split}
\end{equation*}
we denoted
$$v^n=\diver\bigg(\frac{\nabla  u^n}{H(\nabla u^n)}\bigg),\qquad\tilde v^n=\diver\bigg(\frac{\nabla \tilde u^n}{H(\nabla\tilde u^n)}\bigg).$$
Hereafter, times $s,t\in[0,T],\,s\leq t,$ and a continuous function
\begin{equation*}
\gamma:\mathcal{X}_u|_{[0,s]}\times \mathcal{X}_W|_{[0,s]}\longrightarrow [0,1]
\end{equation*}
will be fixed but otherwise arbitrary. The proof is an immediate consequence of the following two lemmas.

\begin{lemma}\label{lemma:wiener}
The process $\tilde{W}$ is a $(\tilde{\mf}_t)$-Wiener process.

\begin{proof}
Obviously, $\tilde{W}$ is a Wiener process and is $(\tilde{\mf}_t)$-adapted. According to the L\'evy martingale characterization theorem, it remains to show that it is also a $(\tilde{\mf}_t)$-martingale.
It holds true
$$\tilde{\stred}\,\gamma\big(\varrho_s \tilde{u}^n,\varrho_s\tilde{W}^n\big)\big[\tilde{W}^n(t)-\tilde{W}^n(s)\big]=\stred\,\gamma\big(\varrho_s u^n,\varrho_s W\big)\big[W(t)-W(s)\big]=0$$
since $W$ is a martingale and the laws of $(\tilde{u}^n,\tilde{W}^n)$ and $(u^n,W)$ coincide.
Next, the uniform estimate
$$\sup_{n\in\mn}\tilde{\stred}|\tilde{W}^n(t)|^2=\sup_{n\in\mn}\stred|W(t)|^2<\infty$$
and the Vitali convergence theorem yields
$$\tilde{\stred}\,\gamma\big(\varrho_s\tilde{u},\varrho_s\tilde{W}\big)\big[\tilde{W}(t)-\tilde{W}(s)\big]=0$$
which finishes the proof.
\end{proof}

\end{lemma}

\begin{lemma}\label{lem:mart}
The processes 
$$\tilde{M},\qquad\tilde{M}^2-\int_0^\tec\big\langle H(\nabla\tilde{u}),\varphi\big\rangle^2\,\dif r,\qquad\tilde{M}\tilde{W}-\int_0^\tec\big\langle H(\nabla \tilde{u}),\varphi\big\rangle\,\dif r$$
are $(\tilde{\mf}_t)$-martingales. 

\begin{proof}
Here, we use the same approach as in the previous lemma. For all $n\in\mn$, the process
$$M^n=\int_0^\tec\big\langle H(\nabla u^n)\,\dif W(r),\varphi\big\rangle$$
is a square integrable $(\mf_t)$-martingale by \eqref{eq:est1} and therefore
$$(M^n)^2-\int_0^\tec\big\langle H(\nabla u^n),\varphi\big\rangle^2\,\dif r,\qquad M^n W-\int_0^\tec\big\langle H(\nabla u^n),\varphi\big\rangle\,\dif r$$
are $(\mf_t)$-martingales. Besides, it follows from the equality of laws that
\begin{equation}\label{exp1}
\begin{split}
&\tilde{\stred}\,\gamma\big(\varrho_s \tilde{u}^n,\varrho_s\tilde{W}^n\big)\big[\tilde{M}^n(t)-\tilde{M}^n(s)\big]\\
&\quad=\stred\,\gamma\big(\varrho_s u^n,\varrho_s W\big)\big[M^n(t)-M^n(s)\big]=0,
\end{split}
\end{equation}
\begin{equation}\label{exp2}
\begin{split}
&\tilde{\stred}\,\gamma\big(\varrho_s \tilde{u}^n,\varrho_s\tilde{W}^n\big)\bigg[(\tilde{M}^n)^2(t)-(\tilde{M}^n)^2(s)-\int_s^t\big\langle H(\nabla\tilde{u}^n),\varphi\big\rangle^2\,\dif r\bigg]\\
&=\stred\,\gamma\big(\varrho_s u^n,\varrho_s W\big)\bigg[(M^n)^2(t)-(M^n)^2(s)-\int_s^t\big\langle H(\nabla u^n),\varphi\big\rangle^2\,\dif r\bigg]=0,
\end{split}
\end{equation}
\begin{equation}\label{exp3}
\begin{split}
&\tilde{\stred}\,\gamma\big(\varrho_s \tilde{u}^n,\varrho_s\tilde{W}^n\big)\bigg[\tilde{M}^n(t)\tilde{W}^n(t)-\tilde{M}^n(s)\tilde{W}^n(s)-\int_s^t\big\langle H(\nabla\tilde{u}^n),\varphi\big\rangle\,\dif r\bigg]\\
&=\stred\,\gamma\big(\varrho_s u^n,\varrho_s W\big)\bigg[M^n(t)W(t)-M^n(s)W(s)-\int_s^t\big\langle H(\nabla u^n),\varphi\big\rangle\,\dif r\bigg]=0.
\end{split}
\end{equation}

In order to pass to the limit in \eqref{exp1}, \eqref{exp2} and \eqref{exp3}, let us first establish the convergence $\tilde M^n(t)\rightarrow\tilde M(t)$ a.s. for all $t\in[0,T]$. Let us only make few comments on the mean curvature term. We recall that according to Proposition \ref{prop:skor} and Corollary \ref{cor:id1} it holds true that
$$\diver\bigg(\frac{\nabla\tilde u^n}{H(\nabla\tilde u^n)}\bigg)\rightharpoonup\diver\bigg(\frac{\nabla\tilde u}{H(\nabla\tilde u)}\bigg)\quad\text{in}\quad L^2(0,T;L^2(\mt)))\quad\tilde\prst\text{-a.s.}$$
Moreover,
$$H(\nabla \tilde u^n)\rightarrow H(\nabla \tilde u)\quad\text{in}\quad L^2(0,T;L^2(\mt)))\quad\tilde\prst\text{-a.s.}$$
and therefore
$$\diver\bigg(\frac{\nabla\tilde u^n}{H(\nabla\tilde u^n)}\bigg)H(\nabla \tilde u^n)\rightharpoonup\diver\bigg(\frac{\nabla\tilde u}{H(\nabla\tilde u)}\bigg)H(\nabla \tilde u)$$
in $L^1(0,T;L^1(\mt)))$ a.s. which yields the desired convergence.

Besides, we observe that according to \eqref{exp1}, \eqref{exp2}, \eqref{exp3} it follows for every $n\in\mn$ that
$$\tilde M^n=\int_0^\tec\big\langle H(\nabla \tilde u^n),\varphi\big\rangle\,\dif \tilde W^n\qquad\tilde\prst\text{-a.s}.$$
Therefore, the passage to the limit in \eqref{exp1} and in the first terms on the left hand side of \eqref{exp2} and \eqref{exp3} (and the same for the right hand side) can be justified by using the convergence $\tilde M^n(t)\rightarrow\tilde M(t)$ together with the uniform integrability given by Proposition \ref{prop:est1}:
\begin{align*}
\tilde\stred\big|\tilde M^n(t)\big|^{p}&\leq C\,\tilde\stred\bigg(\int_0^{t}\big\langle H(\nabla \tilde u^n),\varphi\big\rangle^2\,\dif t\bigg)^{p/2}\\
&\leq C\bigg(1+\tilde \stred\int_0^T\|\nabla \tilde u^n\|_{L^2}^p\,\dif t\bigg)\leq C<\infty.
\end{align*}
This estimate also yields the necessary uniform integrability that together with
$$\big\langle H(\nabla\tilde{u}^n),\varphi\big\rangle\rightarrow\big\langle H(\nabla\tilde{u}^n),\varphi\big\rangle\quad \text{a.e.}\;(\omega,r)$$
justifies the passage to the limit in the remaining terms in \eqref{exp2} and \eqref{exp3} which completes the proof.
\end{proof}
\end{lemma}

\begin{proof}[Proof of Theorem \ref{prop:martsol}]
Once the above lemmas established, we infer that
$$\bigg\langle\!\!\!\bigg\langle\tilde{M}-\int_0^\tec\big\langle H(\nabla\tilde{u})\,\dif\tilde{W},\varphi\big\rangle\bigg\rangle\!\!\!\bigg\rangle=0$$
and consequently \eqref{eq:weak} holds true.
Moreover, we note that the equation \eqref{eq:weak} is in fact satisfied in a stronger sense: since $\tilde u\in L^2(\tilde \Omega;L^2(0,T;H^2(\mt)))$ due to Propositions \ref{prop:est1} and \ref{prop:L2} and $H(\nabla\tilde u)\tilde v\in L^1(\tilde\Omega;L^1(0,T;L^1(\mt)))$ which follows from the proof of Lemma \ref{lem:mart} and therefore \eqref{eq:strong} follows.
%
\end{proof}

\section{Vanishing viscosity limit}\label{sec:van-vis}

The aim of this final section is to study the limit $\varepsilon\rightarrow0$ in \eqref{eq:approx2} and complete the proof of Theorem \ref{thm:main}.
Recall that it was proved in Section \ref{sec:viscous} that for every $\varepsilon\in (0,1)$ there exists
$$\big((\tilde\Omega^\varepsilon,\tilde\mf^\varepsilon,(\tilde\mf^\varepsilon_t),\tilde\prst^\varepsilon),\tilde u^\varepsilon,\tilde W^\varepsilon\big)$$
which is a martingale solution to \eqref{eq:approx2} with the initial law $\Lambda^\varepsilon$. We recall that $\Lambda^\varepsilon\overset{*}{\rightharpoonup}\Lambda$ in the sense of measures on $H^1(\mt)$. It was shown in \cite{jakubow} that it is enough to consider only one probability space, namely,
$$(\tilde\Omega^\varepsilon,\tilde\mf^\varepsilon,\tilde\prst^\varepsilon)=\big([0,1],\mathcal{B}([0,1]),\mathcal{L}\big)\qquad\forall \varepsilon\in (0,1)$$
where $\mathcal{L}$ denotes the Lebesgue measure on $[0,1]$.
Moreover, we can assume without loss of generality that there exists one common Wiener process $W$ for all $\varepsilon$. Indeed, one could perform the compactness argument of the previous section for all the parameters from any chosen subsequence $\varepsilon_n$ at once by redefining
$$\mathcal{X}=\Big(\prod_{n\in\mn}\mathcal{X}_u\Big)\times\mathcal{X}_W$$
and proving tightness of the joint laws of $(u^{\eta,\varepsilon_1},u^{\eta,\varepsilon_2},\dots,W)$ for $\eta\in(0,1)$.
In order to further simplify the notation we also omit the tildas and denote the martingale solution found in Section \ref{sec:viscous} by
$$\big((\Omega,\mf,(\mf_t),\prst),u^\varepsilon,W\big).$$

\subsection{Estimates}

We start with an estimate of the surface area and the mean curvature term. The proof of the following are bounds requires $N=2$. From now on we therefore restrict ourselves to $N=2$ and two-dimensional graphs in $\R^3$.

\begin{prop}\label{prop:area}
For any $\varepsilon>0$ we have the following uniform estimate
\begin{equation}\label{eq:area}
\begin{split}
\stred\sup_{0\leq t\leq T}\int_{\mttwo}H(\nabla u^\varepsilon)\,\dif x+\frac{1}{2}\,\stred\int_0^T\int_{\mttwo}&\bigg|\diver\bigg(\frac{\nabla u^\varepsilon}{H(\nabla u^\varepsilon)}\bigg)\bigg|^2 H(\nabla
u^\varepsilon)\,\dif x\,\dif t\\
&\leq C\,\stred\int_{\mttwo}H(\nabla u^\varepsilon(0))\,\dif x+K(\varepsilon),
\end{split}
\end{equation}
where $K(\varepsilon)\rightarrow 0$ as $\varepsilon\rightarrow0$.

\begin{proof}
We start with some calculations that hold for any $N\in\N$. In the first step, it is necessary to derive the equation satisfied
by $\nabla u^\varepsilon$ and then apply the It\^o formula to the function $p\mapsto\int_{\mt}H(p)\,\dif x$. In order to make the calculation rigorous we make use of the generalized It\^o formula as introduced in \cite[Appendix A]{dehovo}. To be more precise, we consider \eqref{eq:approx33} and obtain
\begin{equation}\label{eq:grad}
\begin{split}
\dif \nabla u^\varepsilon &= \varepsilon\nabla\Delta u^\varepsilon\,\dif t+\nabla\bigg[H\diver\bigg(\frac{\nabla u^\varepsilon}{H}\bigg)\bigg]\dif t\\
&\quad\quad+\frac{1}{2}\nabla\bigg[\frac{(\nabla u^\varepsilon)^*}{H}\totdif^2u^\varepsilon\frac{\nabla u^\varepsilon}{H}\bigg]\dif t+\nabla H\dif W.
\end{split}
\end{equation}
(For notational simplicity we do not stress the dependence of $H$ on $\nabla u^\varepsilon$.) Note that $\nabla u^\varepsilon\in L^2(\Omega;L^2(0,T;H^1(\mt)))$ according to Proposition \ref{prop:est1} hence \eqref{eq:grad} can be rewritten as
$$\dif \nabla u^\varepsilon =\nabla F(t)\dif t+G(t)\dif W$$
where $F,\,G\in L^2(\Omega;L^2(0,T;L^2(\mt)))$. Besides, $H\in C^2(\mr^N)$ has bounded derivatives so the only assumption of \cite[Proposition A.1]{dehovo} which is not satisfied is $\nabla u^\varepsilon\in L^2(\Omega;C([0,T];L^2(\mt)))$. However, following the proof of \cite[Proposition A.1]{dehovo}, one can easily see that under the boundedness hypothesis for $\totdif H$ everything works well even without it.

Therefore, we arrive at the following It\^o formula for the surface measure 
\begin{equation*}
\begin{split}
\dif\int_{\mt}H\,\dif x&=\int_{\mt}\frac{\nabla u}{H}\nabla\big[\varepsilon\Delta u^\varepsilon\big]\,\dif x\,\dif t+\int_{\mt}\frac{\nabla u^\varepsilon}{H}\nabla\Big[H\diver\Big(\frac{\nabla u^\varepsilon}{H}\Big)\Big]\,\dif x\,\dif t\\
&\quad+\frac{1}{2}\int_{\mt}\frac{\nabla u^\varepsilon}{H}\nabla\Big[\frac{(\nabla u^\varepsilon)^*}{H}\totdif^2u^\varepsilon\frac{\nabla u^\varepsilon}{H}\Big]\,\dif x\,\dif 
t+\int_{\mt}\frac{\nabla u^\varepsilon}{H}\nabla H\,\dif x\,\dif W\\
&\quad+\frac{1}{2}\int_{\mt}\frac{1}{H}\Big(\mathrm{Id}-\frac{\nabla 
u^\varepsilon}{H}\otimes\frac{\nabla u^\varepsilon}{H}\Big):\Big(\nabla H\otimes\nabla H\Big)\,\dif x\,\dif t\\
&=J_1+\cdots+J_5,
\end{split}
\end{equation*}
where, for the reader's convenience we recall that in the formulas above  
$\dif x$ stands  for the spatial integration on 
$\mt$ and $\dif t$ respectively $\dif W$ for the time differentials in It\^o sense. 
 It follows from the above consideration that the stochastic integral $J_4$ is a 
square integrable martingale so has zero expectation. After integration by parts 
$J_2$ has a negative sign. For $J_1$ we have
\begin{equation*}
\begin{split}
 J_1\leq\frac{1}{2}\int_{\mt}\Big|\diver\Big(\frac{\nabla u^\varepsilon}{H}\Big)\Big|^2\,H\,\dif x+\frac{\varepsilon^2}{2}\int_{\mt}|\Delta u^\varepsilon|^2\,\dif x.
\end{split}
\end{equation*}
The first term on the right hand side is controlled by $J_2$ whereas the integral over $\Omega\times[0,T]$ of the second one vanishes as $\varepsilon\rightarrow 0$ due to Proposition \ref{prop:est1}.
Next, we will show that $J_3+J_5=0$. Here it is convenient to define $w:=\frac{\nabla u^\varepsilon}{H}$ and to observe that
\begin{align*}
	\frac{1}{H}\Big(\mathrm{Id}-\frac{\nabla u^\varepsilon}{H}\otimes\frac{\nabla u^\varepsilon}{H}\Big):\Big(\nabla H\otimes\nabla H\Big)
	\,&=\, \Big(\frac{1}{H}\Big(\mathrm{Id}-\frac{\nabla u^\varepsilon}{H}\otimes\frac{\nabla u^\varepsilon}{H}\Big) D^2u^\eps w\Big)\cdot\nabla H\\
	&=\, (Dw\, w)\cdot\nabla H.
\end{align*}
Using this equality we deduce that
\begin{align*}
	J_3+J_5&=-\frac{1}{2}\int_{\mt}(\diver w)\big(w\cdot\nabla H\big)\,\dif x\,\dif t
	+\frac{1}{2}\int_{\mt}(Dw\, w)\cdot\nabla H\,\dif x\,\dif t\\
	&=-\frac{1}{2}\int_{\mt}\big((\trace Dw)\id -Dw\big)w\cdot \nabla H\,\dif x\,\dif t.
\end{align*}
We now restrict ourselves to $N=2$. Since $(\trace A)\id -A = (\cof A)^T$ for any $A\in \R^{2\times 2}$, since $\nabla\cdot (Aw)=(\nabla\cdot A^T)\cdot w + A^T:Dw$ and since cofactor matrices are divergence-free the last equality further yields
\begin{align*}
	J_3+J_5 
	\,&=\, -\frac{1}{2}\int_{\mt} \big[(\cof Dw)^Tw\big]\cdot \nabla H\,\dif x\,\dif t\\
	&=\, \frac{1}{2}\int_{\mt} (\cof Dw):Dw\,H\,\dif x\,\dif t.
\end{align*}
Now $(\cof Dw):Dw=2\det Dw$ and we arrive at
\begin{align*}
	 J_3+J_5 \,=\, \int_{\mt} \det Dw\,H\,\dif x\,\dif t.
\end{align*}
The integral on the right-hand side is just the integral over the Gaussian curvature of the graph. 
In fact, we have $K=\frac{1}{H^4}\det D^2u=\det Dw$, since
\begin{align*}
	Dw \,&=\, D\Big(\frac{\nabla u}{H}\Big) \,=\, \frac{1}{H}\Big(\id-\frac{\nabla u}{H}\otimes \frac{\nabla u}{H}\Big)D^2u,\\
	\det Dw \,&=\, \frac{1}{H^2}\cdot \det \Big(\id-\frac{\nabla u}{H}\otimes \frac{\nabla u}{H}\Big) \cdot \det D^2u \,=\, \frac{1}{H^2}\cdot 1\cdot \frac{1}{H^2}\det D^2u,
\end{align*}
where we have used that the matrix $\id-\frac{\nabla u}{H}\otimes \frac{\nabla u}{H}$ has 
eigenvalues $1$ and $\frac{1}{H^2}$.\\
Due to our choice of periodic boundary condition the graph has the topology of a torus and the Gauss--Bonnet Theorem yields that $J_3+J_5=0$. This can also be deduced more directly from the formula 
$$
	(\diver h)({w})\det(D w)=\diver\big(\cof(Dw)^T\cdot h(w)\big),
$$
which can be verified by direct computations. Observing that $H=(1-|w|^2)^{-\frac{1}{2}}$ and setting
\begin{align*}
	h(z)=\frac{1-\sqrt{1-|z|^2}}{2|z|^2}z
\end{align*}
we obtain $\diver h (z)= (1-|z|^2)^{-\frac{1}{2}}$ and therefore
$$
	J_3+J_5\,=\,\int_{\mttwo}\diver\big(\cof(Dw)^T\cdot h(w)\big)\,\dif x\,\dif t=0
$$
and consequently for every $t\in[0,T]$
\begin{equation*}
\begin{split}
\stred\int_{\mttwo}H\big(\nabla u^\varepsilon(t)\big)\,\dif x+\frac{1}{2}\stred\int_0^t\int_{\mttwo}&\bigg|\diver\bigg(\frac{\nabla u^\varepsilon}{H}\bigg)\bigg|^2 H\,\dif x\,\dif s\\
&\qquad\leq \stred\int_{\mttwo}H(\nabla u^\varepsilon(0))\,\dif x+K(\varepsilon).
\end{split}
\end{equation*}
In order to obtain \eqref{eq:area} we proceed similarly, the only difference is in the estimate for the stochastic integral:
\begin{align*}
\stred\sup_{0\leq t\leq T}&\bigg|\int_0^t\int_{\mttwo}\diver\bigg(\frac{\nabla u^\varepsilon}{H}\bigg)H\,\dif x\,\dif W\bigg|\\\
&\leq C\,\stred\bigg(\int_0^T\bigg|\int_{\mttwo}\diver\bigg(\frac{\nabla u^\varepsilon}{H}\bigg)H\,\dif x\bigg|^2\dif t\bigg)^{1/2}\\
&\leq C\,\stred\bigg[\bigg(\sup_{0\leq t\leq T}\int_{\mttwo}H\,\dif x\bigg)\bigg(\int_0^T\int_{\mttwo}\bigg|\diver\bigg(\frac{\nabla u^\varepsilon}{H}\bigg)\bigg|^2H\,\dif x\,\dif t\bigg)\bigg]^{1/2}\\
&\leq \frac{1}{2}\stred\sup_{0\leq t\leq T}\int_{\mttwo}H\,\dif x+ C\,\stred\int_0^T\int_{\mttwo}\bigg|\diver\bigg(\frac{\nabla u^\varepsilon}{H}\bigg)\bigg|^2H\,\dif x\,\dif t\\
&\leq \frac{1}{2}\stred\sup_{0\leq t\leq T}\int_{\mttwo}H\,\dif x+C\,\stred\int_{\mttwo}H(\nabla u^\varepsilon(0))\,\dif x+K(\varepsilon)
\end{align*}
which completes the proof.
\end{proof}
\end{prop}

As a consequence we deduce an estimate for the $L^2$-norm of the solution.

\begin{cor}\label{cor:L2}
For any $\varepsilon>0$ we have the following uniform estimate
\begin{equation*}\label{eq:L2}
\stred\|u^\varepsilon\|_{L^2(0,T;L^2)}\leq C.
\end{equation*}

\begin{proof}
In the first step we show an estimate for the mean value of $u^\varepsilon$ over $\mttwo$ and
then we apply the Poincar\'e inequality. Testing \eqref{eq:approx2} by $\varphi\equiv 1$ we obtain
\begin{equation*}
 \begin{split}
 \dif\int_{\mttwo }u^\varepsilon\,\dif x=\frac{1}{2}\int_{\mttwo}H(\nabla u^\varepsilon)\diver\bigg
(\frac{\nabla u^\varepsilon}{H(\nabla u^\varepsilon)}\bigg)\,\dif x\,\dif t+\int_{\mttwo}H(\nabla u^\varepsilon)\,\dif x\,\dif W.
 \end{split}
\end{equation*}
Since the above stochastic integral is a square-integrable martingale, we apply the Burkholder-Davis-Gundy inequality, Proposition \ref{prop:est1} and Proposition \ref{prop:area} and deduce that
\begin{equation*}
 \begin{split}
 \stred\sup_{0\leq t\leq T}\bigg|\int_{\mttwo} u^\varepsilon(t)\,\dif x\bigg|\leq\stred\bigg|\int_{\mttwo} u^\varepsilon(0)\,\dif x\bigg|+C.
 \end{split}
\end{equation*}
The Poincar\'e inequality yields
\begin{equation*}
\| u^\varepsilon(t)\|_{L^2(0,T;L^2)}\leq C\|\nabla u^\varepsilon(t)\|_{L^2(0,T;L^2)}+\sup_{0\leq t\leq T}\bigg|\int_{\mttwo} u^\varepsilon(t)\,\dif x\bigg|
\end{equation*}
and the claim follows.
\end{proof}
\end{cor}

Finally, we proceed with a uniform estimate for the time derivative of $u^\varepsilon$.

\begin{prop}\label{prop:fractreg}
There exists $s,k>0$ and $p\in[1,\infty)$ such that
\begin{equation*}
\stred\|u^\varepsilon\|_{W^{s,2}(0,T;W^{-k,p})}\leq C.
\end{equation*}

\begin{proof}
In order to estimate the stochastic term, we make use of \cite[Lemma 2.1]{fland} which gives bounds for fractional time derivatives of a stochastic integrals. We obtain for $s\in[0,1/2)$ that
$$\stred\bigg\|\int_0^\tec H(\nabla u^\varepsilon)\,\dif W\bigg\|_{W^{s,2}(0,T;L^2)}^2\leq C\,\stred\int_0^T\|H(\nabla u^\varepsilon)\|_{L^2}^2\dif t\leq C.$$

Since $(u^\varepsilon)$ is bounded in $L^1(\Omega;L^1(0,T;W^{1,2}(\mttwo)))$ we deduce that $(\Delta u^\varepsilon)$ is bounded in $L^1(\Omega;L^1(0,T;W^{-1,2}(\mttwo)))$ and as a consequence
$$\bigg(\Big(\frac{1}{2}+\varepsilon\Big)\int_0^\cdot\Delta u^\varepsilon\,\dif s\bigg)\quad\text{is bounded in}\quad L^1(\Omega;W^{1,1}(0,T;W^{-1,2}(\mttwo))).$$
Regarding the remaining term, we deduce from \eqref{eq:area} that
$$\bigg(H(\nabla u^\varepsilon)\diver\Big(\frac{\nabla u^\varepsilon}{H(\nabla u^\varepsilon)}\Big)\bigg)\quad\text{is bounded in}\quad L^1(\Omega;L^1(0,T;L^1(\mttwo)))$$
hence
$$\bigg(\int_0^\tec H(\nabla u^\varepsilon)\diver\Big(\frac{\nabla u^\varepsilon}{H(\nabla u^\varepsilon)}\Big)\dif s\bigg)\quad\text{is bounded in}\quad L^1(\Omega;W^{1,1}(0,T;L^1(\mttwo))).$$
Altogether, we obtain that $(u^\varepsilon)$ is bounded in $L^1(\Omega;W^{s,2}(0,T;W^{-k,p}))$ where $k,p$ are determined by the Sobolev embedding theorem so that
$$L^1(\mttwo)\hookrightarrow W^{-k,p}(\mttwo),\qquad W^{-1,2}(\mttwo)\hookrightarrow W^{-k,p}(\mttwo)$$
and the proof is complete.
\end{proof}
\end{prop}

\subsection{Compactness}

Let us define the path space $$\mathcal{X}=\mathcal{X}_u\times\mathcal{X}_\kappave\times\mathcal{X}_\nu\times\mathcal{X}_V\times\mathcal{X}_I\times\mathcal{X}_{u_0}\times\mathcal{X}_W,$$ where
$$\mathcal{X}_u=L^2\big(0,T;L^2(\mttwo)\big),\qquad\mathcal{X}_\kappave=\big(L^2(0,T;L^2(\mttwo)),w\big),$$
$$\mathcal{X}_\nu=\big(L^2(0,T;L^2(\mttwo)),w\big),\qquad\mathcal{X}_V=\big(L^2(0,T;L^2(\mttwo)),w\big),$$
$$\mathcal{X}_I=C\big([0,T];\mr\big),\qquad\mathcal{X}_W=C\big([0,T];\mr\big),\qquad\mathcal{X}_{u_0}=H^1(\mttwo).$$
Let us denote by $\mu_{u^\varepsilon}$ the law of $u^\varepsilon$ on $\mathcal{X}_u$, $\eta\in(0,1)$, by $\mu_{\kappave}$, $\mu_{\nu^\varepsilon}$, $\mu_{V^\varepsilon}$ and $\mu_{I^\varepsilon}$, respectively, the law of
$$\kappave:=\diver\bigg(\frac{\nabla u^\varepsilon}{H(\nabla u^\varepsilon)}\bigg),\qquad\nu^\varepsilon:=\frac{\nabla u^\varepsilon}{H(\nabla u^\varepsilon)},$$
$$V^\varepsilon:=\diver\bigg(\frac{\nabla u^\varepsilon}{H(\nabla u^\varepsilon)}\bigg)\sqrt{H(\nabla u^\varepsilon)},\qquad I^\varepsilon:=\int_0^\tec\|H(\nabla u^\varepsilon)\|_{L^1_x}^{1+\theta}\,\dif s$$
on $\mathcal{X}_\kappave$, $\mathcal{X}_\nu$, $\mathcal{X}_V$ and $\mathcal{X}_I$ (for some fixed $\theta\in(0,1)$), respectively, and by $\mu_W$ the law of $W$ on $\mathcal{X}_W$. Recall that the law of $u^\varepsilon(0)$ on $\mathcal{X}_{u_0}$ is given by $\Lambda^\varepsilon$ and due to the construction at the beginning of Section \ref{sec:regul} we immediately obtain tightness of $(\Lambda^\varepsilon)$. The joint law on $\mathcal{X}$ is then denoted by $\mu^\varepsilon$.
In order to prove tightness of $(\mu_{\varepsilon})$, we make use of 
the following compact embedding which can be found in \cite[Corollary 5]{simon}.

\begin{lemma}
Let $X,\,B,\,Y$ be Banach spaces such that
$X\overset{c}{\hookrightarrow}B\hookrightarrow Y.$ If $p,r\in[1,\infty)$ then
$$L^p(0,T;X)\cap W^{s,r}(0,T;Y)\overset{c}{\hookrightarrow} L^p(0,T;B)$$
provided $s>0$ if $r\geq p$ and $s>1/r-1/p$ if $r\leq p.$
\end{lemma}

\begin{prop}
The set of laws $\{\mu^{\varepsilon};\,\varepsilon\in(0,1)\}$ is tight on $\mathcal{X}$.
\begin{proof}
We will show tightness of all the corresponding marginal laws, tightness for the joint laws then follows immediately. Concerning $\{\mu_{u^\varepsilon};\,\varepsilon\in(0,1)\}$, we want to employ the compact embedding
$$L^2(0,T;H^1(\mttwo))\cap W^{s,2}(0,T;W^{-k,p}(\mttwo))\overset{c}{\hookrightarrow} L^2(0,
T;L^2(\mttwo))$$
hence for $R>0$ we define the set
\begin{equation*}
\begin{split}
B_R=&\big\{u\in L^2(0,T;H^1(\mttwo))\cap W^{s,2}(0,T;W^{-k,p}(\mttwo));\\
&\hspace{1cm} \|u\|_{L^2(0,T;H^1(\mttwo))}+\|u\|_{W^{s,2}(0,T;W^{-k,p}(\mttwo))}\leq R\big\}.
\end{split}
\end{equation*}
Now, it holds by Chebyshev inequality, Proposition \ref{prop:est1}, Corollary \ref{cor:L2} and Proposition \ref{prop:fractreg}
\begin{equation*}
\begin{split}
\mu_{u^\varepsilon}(B^c_R)&\leq \prst\bigg(\|u^\varepsilon\|_{L^2(0,T;H^1)}>\frac{R}{2}\bigg)+\prst\bigg(\|u^\varepsilon\|_{W^{s,2}(0,T;W^{-k,p})}>\frac{R}{2}\bigg)\\
&\leq \frac{2}{R}\,\stred\|u^\varepsilon\|_{L^2(0,T;H^1)}+\frac{2}{R}\,\stred\|u^\varepsilon\|_{W^{s,2}(0,T;W^{-k,p})}\leq \frac{C}{R}
\end{split}
\end{equation*}
which yields tightness of $\{\mu_{u^\varepsilon};\,\varepsilon\in(0,1)\}$.

For $\{\mu_{\kappave};\,\varepsilon\in(0,1)\}$, $\{\mu_{\nu^\varepsilon};\,\varepsilon\in(0,1)\}$ and $\{\mu_{V^\varepsilon};\,\varepsilon\in(0,1)\}$) we proceed similarly and make use of the uniform estimate from Proposition \ref{prop:area} together with the fact that
for $R>0$ the set
\begin{equation*}
\begin{split}
B_R=&\big\{z\in L^2(0,T;L^2(\mttwo));\,\|z\|_{L^2(0,T;L^2)}\leq R\big\}.
\end{split}
\end{equation*}
is relatively compact in $\big(L^2(0,T;L^2(\mttwo)),w\big)$.

Regarding $\{\mu_{I^\varepsilon};\,\varepsilon\in(0,1)\}$ we observe that due to Proposition \ref{prop:est1}
$$\big(\|H(\nabla u^\varepsilon)\|_{L^1_x}^{1+\theta}\big)\quad\text{is bounded in}\quad L^{2/{1+\theta}}(\Omega;L^{2/{1+\theta}}(0,T))$$
hence
$$(I^\varepsilon)\quad\text{is bounded in}\quad L^{2/{1+\theta}}(\Omega;W^{1,2/{1+\theta}}(0,T))$$
and due to Sobolev imbedding theorem
$$W^{1,2/{1+\theta}}(0,T)\overset{c}{\hookrightarrow} C([0,T];\mr).$$
Therefore, we obtain tightness of $\{\mu_{I^\varepsilon};\,\varepsilon\in(0,1)\}$ on $\mathcal{X}_I$ and the corresponding tightness of $\mu_W$ follows by the same reasoning as in Proposition \ref{tight}.
\end{proof}
\end{prop}

We apply the Jakubowski-Skorokhod representation theorem and obtain a weakly convergent subsequence $\mu^n=\mu^{\varepsilon_n}$ together with a limit law $\mu$ such that the following result holds true.

\begin{prop}\label{prop:skorokhod}
There exists a probability space $(\tilde{\Omega},\tilde{\mf},\tilde{\prst})$ with a sequence of $\mathcal{X}$-valued random variables $(\tilde{u}^n,\tilde\kappav^n,\tilde\nu^n,\tilde V^n,\tilde I^n,\tilde{W}^n,\tilde u^n_0)$, $n\in\mn,$ and $(\tilde{u},\tilde\kappav,\tilde\nu,\tilde V,\tilde I,\tilde{W},\tilde u_0)$ such that
\begin{enumerate}
 \item the laws of $(\tilde{u}^n,\tilde\kappav^n,\tilde\nu^n,\tilde V^n,\tilde I^n,\tilde{W}^n,\tilde u^n_0)$ and $(\tilde{u},\tilde\kappav,\tilde\nu,\tilde V,\tilde I,\tilde{W},\tilde u_0)$ under $\,\tilde{\prst}$ coincide with $\mu^n$ and $\mu$, respectively,
 \item\label{item:skorokhod} $(\tilde{u}^n,\tilde\kappav^n,\tilde\nu^n,\tilde V^n,\tilde I^n,\tilde{W}^n,\tilde u^n_0)$ converges $\,\tilde{\prst}$-a.s. to $(\tilde{u},\tilde\kappav,\tilde\nu,\tilde V,\tilde I,\tilde{W},\tilde u_0)$ in the topology of $\mathcal{X}$.
\end{enumerate}

\end{prop}

We are immediately able to identify the approximations $\tilde\kappav^n,\tilde\nu^n,\tilde V^n,\tilde I^n$, $n\in\mn$.

\begin{lemma}
For every $n\in\mn$ it holds true a.s.
\begin{align*}
\tilde \kappav^n&=\diver\bigg(\frac{\nabla \tilde u^n}{H(\nabla \tilde u^n)}\bigg),&\tilde \nu^n&=\frac{\nabla \tilde u^n}{H(\nabla \tilde u^n)},\\
\tilde V^n&=\diver\bigg(\frac{\nabla \tilde u^n}{H(\nabla \tilde u^n)}\bigg)\sqrt{H(\nabla \tilde u^n)},&\tilde I^n&=\int_0^\tec\|H(\nabla\tilde u^n)\|_{L^1_x}^{1+\theta}\,\dif s.
\end{align*}

\begin{proof}
According to our energy estimates and in particular due to Proposition \ref{prop:est1} and the surface area estimate from Proposition \ref{prop:area}, the mappings
$$\supp \mu_{u^n}\rightarrow L^2(0,T;L^2(\mttwo)),\;u\mapsto \diver\bigg(\frac{\nabla u}{H(\nabla u)}\bigg),$$
$$\supp \mu_{u^n}\rightarrow L^2(0,T;L^2(\mttwo)),\;u\mapsto \frac{\nabla u}{H(\nabla u)},$$
$$\supp \mu_{u^n}\rightarrow L^2(0,T;L^2(\mttwo)),\;u\mapsto \diver\bigg(\frac{\nabla u}{H(\nabla u)}\bigg)\sqrt{H(\nabla u)}$$
and
$$\supp \mu_{u^n}\rightarrow C([0,T];\mr),\;u\mapsto \int_0^\tec\|H(\nabla u)\|_{L^1_x}^{1+\theta}\,\dif s$$
are well-defined and measurable. Therefore, the claim follows directly from the equality of joint laws of $(u^n,\kappav^n,\nu^n,V^n,I^n)$ and $(\tilde u^n,\tilde\kappav^n,\tilde\nu^n,\tilde V^n,\tilde I^n)$.
\end{proof}
\end{lemma}

As a consequence of the a.s. convergence $\tilde u^n\rightarrow \tilde u$ in $L^2(0,T;L^2(\mttwo))$ and the uniform bound in Proposition \ref{prop:est1} we deduce that
\begin{equation}\label{eq:weaka}
\nabla \tilde u^n\rightharpoonup\nabla\tilde u\quad\text{in}\quad L^2(\tilde\Omega;L^2(0,T;L^2(\mttwo))).
\end{equation}
Nevertheless, as our model problem is nonlinear in $\nabla\tilde u$ it is crucial to establish the strong convergence in order to be able to pass to the limit.

\begin{prop}\label{prop:young}
For all $p\in[1,2)$, it holds true that
$$\nabla \tilde u^n\rightarrow\nabla\tilde u\quad\text{in}\quad L^p(\tilde\Omega;L^p(0,T;L^p(\mttwo))).$$

\begin{proof}
{\em Step 1:} Due to the weak convergence \eqref{eq:weaka}, there exists a Young measure associated to the sequence $(\nabla \tilde u^n)$, i.e. there exists $\sigma:\tilde\Omega\times[0,T]\times\mttwo\rightarrow\mathcal{P}_1(\mr^2)$, where $\mathcal{P}_1(\mr^2)$ denotes the set of probability measures on $\mr^2$, such that for every $B\in C(\mr^2)$ with linear growth
$$B(\nabla \tilde u^n)\rightharpoonup \bar B\quad\text{in}\quad L^2(\tilde\Omega; L^2(0,T;L^2(\mttwo)))$$
where
$$\bar B(t,x)=\langle\sigma_{t,x}, B\rangle\quad\text{a.e.}$$
We refer the reader to \cite{malek} for a thorough exposition of the concept of Young measures, the above applied result can be found in \cite[Theorem 4.2.1, Corollary 4.2.10]{malek}.
The desired strong convergence of $\nabla\tilde u^n$ will be shown once we prove that for a.e. $\omega,t,x$ the Young measure $\sigma$ is a Dirac mass.

{\em Step 2:} In this part of the proof, we show that the following relation holds true a.e.
\begin{equation}\label{eq:young}
\int_{\mr^2}\frac{|p|^2}{\sqrt{1+|p|^2}}\,\dif\sigma_{t,x}(p)=\bigg(\int_{\mr^2}p\,\dif\sigma_{t,x}(p)\bigg)\cdot\bigg(\int_{\mr^2}\frac{p}{\sqrt{1+|p|^2}}\,\dif\sigma_{t,x}(p)\bigg).
\end{equation}
Towards this end, we observe that due to Proposition \ref{prop:skorokhod},
$$\tilde\kappav^n\rightharpoonup\tilde\kappav,\quad\tilde\nu^n\rightharpoonup\tilde\nu\quad\text{in}\quad L^2(0,T;L^2(\mttwo))\quad\text{a.s.}$$
Using $|\tilde\nu^n|\leq 1$ and the Vitali convergence Theorem we also deduce that $\tilde\nu\in L^2(\tilde{\Omega};L^2(0,T;L^2(\mttwo)))$ with
\begin{align*}
	\tilde\nu^n\rightharpoonup\tilde\nu\quad\text{in}\quad L^2(\tilde{\Omega};L^2(0,T;L^2(\mttwo)))
\end{align*}
Besides, $\tilde\nu^n$ is a continuous and bounded function of $\nabla \tilde u^n$ hence, according to {\em Step 1}, $\tilde\nu$ is given by
$$\tilde\nu(t,x)=\int_{\mr^2}\frac{p}{\sqrt{1+|p|^2}}\,\dif\sigma_{t,x}(p).$$

Using integration by parts, it follows easily that $\tilde\kappav=\diver\tilde\nu$ almost everywhere.
Thus, on the one hand, we employ the Div-Curl Lemma type argument from \cite[Theorem 3.1, (3.13)]{EvSp95} and obtain
$$\nabla \tilde u^n\cdot\tilde\nu^n\rightharpoonup\nabla u\cdot\tilde\nu\quad\text{in}\quad L^2(0,T;L^2(\mttwo))\quad\text{a.s.}$$
and consequently by the Vitali convergence theorem
$$\nabla \tilde u^n\cdot\tilde\nu^n\rightharpoonup\nabla u\cdot\tilde\nu\quad\text{in}\quad L^2(\tilde\Omega;L^2(0,T;L^2(\mttwo))).$$
On the other hand, we deduce from {\em Step 1} that the weak limit of $\nabla \tilde u^n\cdot\tilde\nu^n$ is also given by
$$\int_{\mr^2}\frac{|p|^2}{\sqrt{1+|p|^2}}\,\sigma_{t,x}(p)$$
and \eqref{eq:young} follows.

{\em Step 3:} Next, we will infer from \eqref{eq:young} that $\sigma$ reduces to a Dirac mass for a.e. $\omega,t,x$. To simplify the notation, let us denote $f(p)=p,$ $g(p)=\frac{p}{\sqrt{1+|p|^2}}$. Then \eqref{eq:young} reads as
\begin{equation}\label{eq:product}
\langle\sigma, f\cdot g\rangle=\langle \sigma,f\rangle\cdot\langle\sigma,g\rangle.
\end{equation}
Since $\langle\nu,1\rangle=1$, the left hand side of \eqref{eq:product} can be rewritten as
\begin{equation*}
\begin{split}
\frac{1}{2}\bigg(\int_{\mr^2}f(p)\cdot g(p)&\dif\sigma(p)\int_{\mr^2}\dif \sigma (q)+\int_{\mr^2}f(q)\cdot g(q)\dif\sigma( q)\int_{\mr^2}\dif\sigma( p)\bigg)\\
&=\frac{1}{2}\int_{\mr^2}\big(f(p)\cdot g(p)+f(q)\cdot g(q)\big)\dif\sigma\otimes\sigma(p,q)
\end{split}
\end{equation*}
whereas for the right hand side, we have
\begin{equation*}
\begin{split}
\frac{1}{2}\bigg(\int_{\mr^2}f(p)&\dif\sigma(p)\cdot\int_{\mr^2}g(q)\dif\sigma(q)+\int_{\mr^2}f(q)\dif\sigma( q)\cdot\int_{\mr^2}g(p)\dif\sigma( p)\bigg)\\
&=\frac{1}{2}\int_{\mr^2}\big(f(p)\cdot g(q)+f(p)\cdot g(q)\big)\dif\sigma\otimes\sigma(p,q).
\end{split}
\end{equation*}
Thus subtracting the right hand side from the left hand side we deduce that
$$\int_{\mr^2}\big(f(p)-f(q)\big)\cdot\big(g(p)-g(q)\big)\dif\sigma\otimes\sigma(p,q)=0.$$
To conclude, we first observe that
\begin{equation}\label{eq:F}
\begin{split}
F(p,q)&=\big(f(p)-f(q)\big)\cdot\big(g(p)-g(q)\big)> 0\qquad\forall p,q\in\mr^2,\,p\neq q,\\
F(p,p)&=0.
\end{split}
\end{equation}
This follows from the strict convexity of the function $G$, $G(p):=\sqrt{1+|p|^2}$, which is equivalent to the strict monotonicity of $G'=g$. Since $f$ is the identity this proves \eqref{eq:F}.

As a consequence, the support of $\sigma$ needs to be a single point hence necessarily $\sigma_{t,x}=\delta_{\nabla\tilde u(t,x)}$ almost everywhere. 

{\em Step 4:} Since a Young measure being Dirac is equivalent to the convergence in measure we conclude by making use of the a priori estimate from Proposition \ref{prop:est1}.
\end{proof}
\end{prop}

Note that in particular we have proved that
$$H(\nabla\tilde u^n)\rightharpoonup H(\nabla\tilde u )\quad\text{in}\quad L^2(\tilde\Omega;L^2(0,T;L^2(\mathbb T^2)).$$
and that for all $p\in[1,2)$
$$H(\nabla\tilde u^n)\rightarrow H(\nabla\tilde u )\quad\text{in}\quad L^p(\tilde\Omega;L^p(0,T;L^p(\mathbb T^2)).$$
(Such a convergence of the area measures is a crucial property also in many related results for deterministic mean curvature flow, see for example \cite{LuSt95}.) As a consequence, we are able to identify the limits $\tilde V$ and $\tilde I$.

\begin{cor}\label{cor:id}
It holds true a.s.
\begin{align*}
\tilde V&=\diver\bigg(\frac{\nabla \tilde u}{H(\nabla \tilde u)}\bigg)\sqrt{H(\nabla \tilde u)},&\tilde I&=\int_0^\tec\|H(\nabla\tilde u)\|_{L^1_x}^{1+\theta}\,\dif s.
\end{align*}

\begin{proof}
In order to identify the limit of $\tilde V^n$, observe that due to Proposition \ref{prop:young}, for all $q\in[1,\infty)$
$$\frac{\nabla \tilde u^n}{H(\nabla \tilde u^n)}\rightarrow \frac{\nabla \tilde u}{H(\nabla \tilde u)}\quad\text{in}\quad L^q(\tilde\Omega;L^q(0,T;L^q(\mttwo)))$$
hence according to Proposition \ref{prop:area}
$$\diver\bigg(\frac{\nabla \tilde u^n}{H(\nabla \tilde u^n)}\bigg)\rightharpoonup\diver\bigg( \frac{\nabla \tilde u}{H(\nabla \tilde u)}\bigg)\quad\text{in}\quad L^2(\tilde\Omega;L^2(0,T;L^2(\mttwo))).$$
Besides,
\begin{equation}\label{eq:sqrtH}
\sqrt{H(\nabla\tilde u^n)}\rightarrow \sqrt{H(\nabla\tilde u)}\quad\text{in}\quad  L^2(\tilde\Omega;L^2(0,T;L^2(\mttwo)))
\end{equation}
and consequently
$$\diver\bigg(\frac{\nabla \tilde u^n}{H(\nabla \tilde u^n)}\bigg)\sqrt{H(\nabla\tilde u^n)}\rightharpoonup\diver\bigg( \frac{\nabla \tilde u}{H(\nabla \tilde u)}\bigg)\sqrt{H(\nabla\tilde u)}$$
in $ L^1(\tilde\Omega;L^1(0,T;L^1(\mttwo)))$ which gives the identification of $\tilde V$.

Identification of $\tilde I$ follows from the fact that for every $t\in[0,T]$
$$\int_0^t\|H(\nabla\tilde u^n)\|_{L^1_x}^{1+\theta}\,\dif s\rightarrow \int_0^t\|H(\nabla\tilde u)\|_{L^1_x}^{1+\theta}\,\dif s$$
according to Proposition \ref{prop:young}, Proposition \ref{prop:est1} and the Vitali convergence theorem.
\end{proof}
\end{cor}

\subsection{Identification of the limit}\label{subsec:5.3}

Let $(\tilde{\mf}_t)$ be the $\tilde{\prst}$-augmented canonical filtration of the process $(\tilde{u},\tilde{W},\tilde u_0)$. Note that $\tilde V$ and $\tilde I$ are adapted to $(\tilde\mf_t)$ as well due to Corollary \ref{cor:id}.
Now everything is prepared to establish the final existence result, which in particular proves the main Theorem \ref{thm:main}.

\begin{thm}\label{prop:martsol1}
$\big((\tilde{\Omega},\tilde{\mf},(\tilde{\mf}_t),\tilde{\prst}),\tilde u,\tilde{W}\big)$
is a weak martingale solution to \eqref{eq:smcf2} with the initial law $\Lambda$. That is, it satisfies Definition \ref{def:sol} and in particular for every $\varphi\in C^\infty(\mttwo)$ it holds true for a.e. $t\in[0,T]$ a.s. that
\begin{equation}\label{eq:weak1}
 \begin{split}
  \langle \tilde u(t),\varphi\rangle&=\langle \tilde u_0,\varphi\rangle+\frac{1}{2}\int_0^t\langle \tilde u,\Delta\varphi\rangle\dif s+\frac{1}{2}\int_0^t\langle \sqt\tilde v,\varphi\rangle\dif s\\
&\qquad+\int_0^t\langle \sqt\dif\tilde W,\varphi\rangle,
 \end{split}
\end{equation}
where
\begin{equation}\label{eq:v1}
\tilde v=\diver\bigg(\frac{\nabla \tilde u}{H(\nabla \tilde u)}\bigg).
\end{equation}
\end{thm}

The proof is based on a refined identification limit procedure which in comparison to Subsection \ref{subsec:identif} includes two new ingredients.
First, the method of densely defined martingales which was developed in \cite{hof} is applied in order to deal with martingales that are only defined for almost all times and no continuity properties are a priori known (see \cite[Theorem 4.13, Appendix]{hof}). In that case, the corresponding quadratic variations are not well defined and the approach of Subsection \ref{subsec:identif} does not apply directly.
Second, the local martingales approach of \cite{hofse} is invoked to overcome the difficulty in the passage to the limit.

Both issues originate in the lack of uniform moment estimates for $\nabla u^n$. Indeed, on the one hand, we are not able to obtain tightness of $(u^n)$ in any space of continuous (or weakly continuous) functions in time and consequently the passage to the limit in the corresponding martingales can be performed only for a.e. $t\in[0,T]$.
On the other hand, we are only able to establish the strong convergence $\nabla\tilde u^n\rightarrow \nabla\tilde u$ in $L^p(0,T;L^p(\mttwo))$ a.s. for $p\in[1,2)$ and the convergence in $L^2(0,T;L^2(\mttwo))$ remains weak, which is not enough to pass to the limit in the quadratic variation. Note that the problem lies in particular in the weak convergence with respect to time rather than space as we consider weak solutions in $x$ anyway.

We claim that as a consequence of Proposition \ref{prop:skorokhod}, it holds true that
\begin{equation}\label{eq:convinmeas}
\tilde u^n\rightarrow\tilde u\quad\text{ in }\quad L^2(\mttwo)\quad\text{ in measure }\quad \tilde\prst\otimes\mathcal{L}_{[0,T]}
\end{equation}
and consequently there exists $\mathcal{D}\subset[0,T]$ of full Lebesgue measure such that (up to subsequence)
\begin{equation}\label{convinmeas2}
\tilde u^n(t)\rightarrow\tilde u(t)\quad\text{ in }\quad L^2(\mttwo)\quad\tilde\prst\text{-a.s.}\quad\forall t\in\mathcal{D}.
\end{equation}
Indeed, \eqref{eq:convinmeas} follows directly from the dominated convergence theorem since for every $\delta\in(0,1)$
\begin{align*}
\tilde\prst\otimes\mathcal{L}_{[0,T]}&\Big(\|\tilde u^n-\tilde u\|_{L^2_x}>\delta\Big)=\tilde\stred\int_0^T\ind_{\{\|\tilde u^n(t)-\tilde u(t)\|_{L^2_x}>\delta\}}\,\dif t
\end{align*}
where for a.e. $\omega$ the inner integral converges to $0$ due to Proposition \ref{prop:skorokhod}.

Note that $\mathcal{D}$ is dense in $[0,T]$ since it is complement of a set with zero Lebesgue measure.
For all $t\in\mathcal{D}$ and a test function $\varphi\in C^\infty(\mttwo)$ we define
\begin{equation*}
\begin{split}
M^n(t)&=\big\langle u^n(t),\varphi\big\rangle-\big\langle u^n(0),\varphi\big\rangle-\Big(\frac{1}{2}+\varepsilon_n\Big)\int_0^t\langle u^n,\Delta\varphi\rangle\dif s-\frac{1}{2}\int_0^t\langle H(\nabla u^n) v^n,\varphi\rangle\dif s,\\
\tilde{M}^n(t)&=\big\langle \tilde u^n(t),\varphi\big\rangle-\big\langle \tilde u^n_0,\varphi\big\rangle-\Big(\frac{1}{2}+\varepsilon_n\Big)\int_0^t\langle \tilde u^n,\Delta\varphi\rangle\dif s-\frac{1}{2}\int_0^t\langle H(\nabla \tilde u^n)\tilde v^n,\varphi\rangle\dif s,\\
\tilde{M}(t)&=\big\langle \tilde u(t),\varphi\big\rangle-\big\langle \tilde u_0,\varphi\big\rangle-\frac{1}{2}\int_0^t\langle \tilde u,\Delta\varphi\rangle\dif s-\frac{1}{2}\int_0^t\langle \sqt\tilde v,\varphi\rangle\dif s,
\end{split}
\end{equation*}
and recall that
$$v^n=\diver\bigg(\frac{\nabla  u^n}{H(\nabla u^n)}\bigg),\qquad\tilde v^n=\diver\bigg(\frac{\nabla \tilde u^n}{H(\nabla\tilde u^n)}\bigg).$$

\begin{prop}\label{prop:mart1}
The process $\tilde{W}$ is a $(\tilde{\mf}_t)$-Wiener process, the processes 
$$\tilde{M},\qquad\tilde{M}^2-\int_0^\tec\big\langle H(\nabla\tilde{u}),\varphi\big\rangle^2\,\dif r,\qquad\tilde{M}\tilde{W}-\int_0^\tec\big\langle H(\nabla \tilde{u}),\varphi\big\rangle\,\dif r,$$
indexed by $t\in\mathcal{D}$, are $(\tilde{\mf}_t)$-local martingales.

\begin{proof}
The first claim follows immediately by the same reasoning as in Lemma \ref{lemma:wiener}. To prepare the proof of the remaining parts, let $R\in\mr^+$ and define
$$\tau_R:C([0,T];\mr)\rightarrow [0,T],\quad f\mapsto \inf\big\{t>0; |f(t)|\geq R\big\}.$$
(with the convention $\inf\emptyset=T$).
Then for every $I^n$, one may use Proposition \ref{prop:est1} and deduce that $\tau_R(I^n)$ defines an $(\mf_t)$-stopping time and the blow up does not occur in a finite time, i.e.
\begin{equation}\label{eq:nonblowup}
\sup_{R\in\mr^+}\tau_R(I^n)=T\qquad\text{a.s.}
\end{equation}
The same is valid for the case of $\tilde I^n$ and $\tilde I$.
The stopping times $\tau_R(\tilde I)$ will play the role of a localizing sequence for the processes
$$\tilde{M},\qquad\tilde{M}^2-\int_0^\tec\big\langle H(\nabla\tilde{u}),\varphi\big\rangle^2\,\dif r,\qquad\tilde{M}\tilde{W}-\int_0^\tec\big\langle H(\nabla \tilde{u}),\varphi\big\rangle\,\dif r.$$
In particular, we employ  $\tau_R(\tilde I^n)$ as a localizing sequence for the approximations
$$\tilde{M}^n,\qquad(\tilde{M}^n)^2-\int_0^\tec\big\langle H(\nabla\tilde{u}^n),\varphi\big\rangle^2\,\dif r,\qquad\tilde{M}^n\tilde{W}^n-\int_0^\tec\big\langle H(\nabla \tilde{u}^n),\varphi\big\rangle\,\dif r$$
and pass to the limit.
Therefore, it is also necessary to establish the convergence of the stopping times, that is, for a fixed $R\in\mr^+$ we need to verify
$$\tau_R(\tilde I^n)\rightarrow \tau_R(\tilde I)\qquad \text{a.s.}$$
so it is a question of continuity of $\tau_R(\cdot)$.
This is not true in general but due to observations made in \cite[Lemma 3.5, Lemma 3.6]{hofse}, there exists a sequence $R_m\rightarrow\infty$ such that
\begin{equation}\label{eq:cont}
\tilde\prst\big(\tau_{R_m}(\cdot) \;\text{is continuous at}\; \tilde I\big)=1
\end{equation}
and in the sequel we only employ $R_m$ from this sequence.

Let us proceed with the proof. We observe that, for all $n\in\mn$, the process
$$M^n=\int_0^\tec\big\langle H(\nabla u^n)\,\dif W(r),\varphi\big\rangle$$
is a square integrable $(\mf_t)$-martingale by \eqref{eq:est1} and therefore
$$(M^n)^2-\int_0^\tec\big\langle H(\nabla u^n),\varphi\big\rangle^2\,\dif r,\qquad M^n W-\int_0^\tec\big\langle H(\nabla u^n),\varphi\big\rangle\,\dif r$$
are $(\mf_t)$-martingales. Therefore, as in Lemma \ref{lem:mart}, we obtain for fixed $n\in\mn$ from the equality of laws that
\begin{equation}\label{exp11}
\begin{split}
\tilde{\stred}&\,\gamma\big(\varrho_s \tilde{u}^n,\varrho_s\tilde{W}^n,\tilde u^n_0\big)\big[\tilde{M}^n\big(t\wedge\tau_{R_m}(\tilde I^n)\big)\big]\\
=&\,\tilde{\stred}\,\gamma\big(\varrho_s \tilde{u}^n,\varrho_s\tilde{W}^n,\tilde u^n_0\big)\big[\tilde{M}^n\big(s\wedge\tau_{R_m}(\tilde I^n)\big)\big],
\end{split}
\end{equation}
\begin{equation}\label{exp21}
\begin{split}
\tilde{\stred}&\,\gamma\big(\varrho_s \tilde{u}^n,\varrho_s\tilde{W}^n,\tilde u^n_0\big)\bigg[(\tilde{M}^n)^2\big(t\wedge\tau_{R_m}(\tilde I^n)\big)-\int_{0\wedge\tau_{R_m}(\tilde I^n)}^{t\wedge\tau_{R_m}(\tilde I^n)}\big\langle H(\nabla\tilde{u}^n),\varphi\big\rangle^2\,\dif r\bigg]\\
=&\,\tilde{\stred}\,\gamma\big(\varrho_s \tilde{u}^n,\varrho_s\tilde{W}^n,\tilde u^n_0\big)\bigg[(\tilde{M}^n)^2\big(s\wedge\tau_{R_m}(\tilde I^n)\big)-\int_{0}^{s\wedge\tau_{R_m}(\tilde I^n)}\big\langle H(\nabla\tilde{u}^n),\varphi\big\rangle^2\,\dif r\bigg],
\end{split}
\end{equation}
\begin{equation}\label{exp31}
\begin{split}
\tilde{\stred}&\,\gamma\big(\varrho_s \tilde{u}^n,\varrho_s\tilde{W}^n,\tilde u^n_0\big)\bigg[\tilde{M}^n\tilde{W}^n\big(t\wedge\tau_{R_m}(\tilde I^n)\big)-\int_0^{t\wedge\tau_{R_m}(\tilde I^n)}\big\langle H(\nabla\tilde{u}^n),\varphi\big\rangle\,\dif r\bigg]\\
=&\,\tilde{\stred}\,\gamma\big(\varrho_s \tilde{u}^n,\varrho_s\tilde{W}^n,\tilde u^n_0\big)\bigg[\tilde{M}^n\tilde{W}^n\big(s\wedge\tau_{R_m}(\tilde I^n)\big)-\int_0^{s\wedge\tau_{R_m}(\tilde I^n)}\big\langle H(\nabla\tilde{u}^n),\varphi\big\rangle\,\dif r\bigg],
\end{split}
\end{equation}
where $s,t\in[0,T],\,s\leq t,$ and
$$\gamma:\mathcal{X}_u|_{[0,s]}\times\mathcal{X}_W|_{[0,s]}\times \mathcal{X}_{u_0}\rightarrow [0,1]$$
is a continuous function.

In order to pass to the limit in \eqref{exp11}, \eqref{exp21} and \eqref{exp31}, let us first establish the convergence $\tilde M^n(t)\rightarrow\tilde M(t)$ a.s. for all $t\in\mathcal{D}$. Concerning the term $\langle\tilde u^n(t),\varphi\rangle$ we conclude immediately due to \eqref{convinmeas2}. Since convergence of the third term in $\tilde M^n(t)$ follows directly from Proposition \ref{prop:skorokhod}, let us proceed with the mean curvature term. We recall that according to Proposition \ref{prop:skorokhod} and Corollary \ref{cor:id} it holds true that
$$\diver\bigg(\frac{\nabla\tilde u^n}{H(\nabla\tilde u^n)}\bigg)\sqrt{H(\nabla \tilde u^n)}\rightharpoonup\diver\bigg(\frac{\nabla\tilde u}{H(\nabla\tilde u)}\bigg)\sqrt{H(\nabla \tilde u)}$$
in $L^2(0,T;L^2(\mttwo)))$ almost surely. Moreover, in view of \eqref{eq:sqrtH} we obtain
$$\diver\bigg(\frac{\nabla\tilde u^n}{H(\nabla\tilde u^n)}\bigg)H(\nabla \tilde u^n)\rightharpoonup\diver\bigg(\frac{\nabla\tilde u}{H(\nabla\tilde u)}\bigg)H(\nabla \tilde u)$$
in $L^1(0,T;L^1(\mttwo)))$ almost surely. which yields the desired convergence of the corresponding term in $\tilde M^n(t)$.

Moreover, we observe that according to \eqref{exp11}, \eqref{exp21}, \eqref{exp31} and \cite[Proposition A.1]{hof} it follows for every $n\in\mn$ that
$$\tilde M^n=\int_0^\tec\big\langle H(\nabla \tilde u^n),\varphi\big\rangle\,\dif \tilde W^n\qquad\forall t\in\mathcal{D}\qquad\tilde\prst\text{-a.s}.$$
Therefore, the passage to the limit in \eqref{exp11} and in the first terms on the left hand side of \eqref{exp21} and \eqref{exp31} (and the same for the right hand side) can be justified by using the convergence $\tilde M^n(t)\rightarrow\tilde M(t)$ together with the uniform integrability given by
\begin{align*}
\tilde\stred&\big|\tilde M^n\big(t\wedge\tau_{R_m}( \tilde I^n)\big)\big|^{2+\kappav}\leq C\,\tilde\stred\bigg(\int_0^{\tau_{R_m}( \tilde I^n)}\big\langle H(\nabla \tilde u^n),\varphi\big\rangle^2\,\dif t\bigg)^{(2+\kappav)/2}\\
&\leq C\,\tilde \stred\bigg[\sup_{0\leq t\leq T}\|H(\nabla\tilde{u}^n)\|_{L^1_x}\int_0^{\tau_{R_m}( \tilde I^n)}\|H(\nabla\tilde{u}^n)\|_{L^1_x}^{1+\theta}\,\dif r\bigg]\leq C_\delta R_m.
\end{align*}
This estimate also yields the necessary uniform integrability that together with
$$\big\langle H(\nabla\tilde{u}^n),\varphi\big\rangle\rightarrow\big\langle H(\nabla\tilde{u}^n),\varphi\big\rangle\quad \text{a.e.}\;(\omega,r)$$
justifies the passage to the limit in the remaining terms in \eqref{exp21} and \eqref{exp31}. Thus we have shown that $\tilde{M}^2-\int_0^\tec\big\langle H(\nabla\tilde{u}),\varphi\big\rangle^2\,\dif r$ and $\tilde{M}\tilde{W}-\int_0^\tec\big\langle H(\nabla \tilde{u}),\varphi\big\rangle\,\dif r$ are densely defined local martingales with respect to $(\tilde\mf_t)$ and the proof is complete.
\end{proof}
\end{prop}

\begin{proof}[Proof of Theorem \ref{prop:martsol1}]
Having Proposition \ref{prop:mart1} in hand, we apply \cite[Proposition A.1]{hof} for the stopped processes 
\begin{align*}
&\tilde{M}\big(\cdot\wedge\tau_{R_m}(\tilde I)\big),\qquad\tilde{M}^2\big(\cdot\wedge\tau_{R_m}(\tilde I)\big)-\int_0^{\cdot\wedge\tau_{R_m}(\tilde I)}\big\langle H(\nabla\tilde{u}),\varphi\big\rangle^2\,\dif r,\\
&\tilde{M}\tilde{W}\big(\cdot\wedge\tau_{R_m}(\tilde I)\big)-\int_0^{\cdot\wedge\tau_{R_m}(\tilde I)}\big\langle H(\nabla \tilde{u}),\varphi\big\rangle\,\dif r,
\end{align*}
and deduce that
$$\tilde{M}\big(\cdot\wedge\tau_{R_m}(\tilde I)\big)=\int_0^{\cdot\wedge\tau_{R_m}(\tilde I)}\big\langle H(\nabla\tilde{u})\,\dif\tilde{W},\varphi\big\rangle\qquad\forall t\in\mathcal{D}\quad\tilde\prst\text{-a.s.}$$
for every $m\in\mn$ and consequently \eqref{eq:weak1} holds true due to \eqref{eq:nonblowup}.

In particular, $\tilde M$ can be defined for all $t\in[0, T ]$ such that it has a modification which is a continuous $(\tilde\mf_t)$-local martingale and furthermore, due to Proposition \ref{prop:est1}, it is a $(\tilde\mf_t)$-martingale. Besides, we observe that \eqref{eq:smcf2} is satisfied in $H^{-1}(\mttwo)$ and, as a consequence, $\tilde u$ (as a class of equivalence) has a representative with almost surely continuous trajectories in $H^{-1}(\mttwo)$ and hence is measurable with respect to the predictable $\sigma$-field $\mathcal{P}$. The continuous embedding $H^{1}(\mttwo)\hookrightarrow H^{-1}(\mttwo)$ then implies that $\tilde u\in L^2(\tilde\Omega\times[0,T],\mathcal{P},\dif\prst\otimes\dif t;H^1(\mttwo))$ as required by Definition \ref{def:sol}. Indeed, any Borel subset of $H^{1}(\mttwo)$ is also Borel in $H^{-1}(\mttwo)$ and therefore its preimage under $\tilde u$ is predictable. The proof is complete.
\end{proof}


\begin{thebibliography}{29}

\bibitem{AlCa79}
S.~M. Allen and J.~W. Cahn.
\newblock A microscopic theory for antiphase boundary motion and its
  application to antiphase domain coarsening.
\newblock {\em Acta Metall.}, 27:1085–1095, 1979.

\bibitem{AlTW93}
Fred Almgren, Jean~E. Taylor, and Lihe Wang.
\newblock Curvature-driven flows: a variational approach.
\newblock {\em SIAM J. Control Optim.}, 31(2):387--438, 1993.

\bibitem{BaSS93}
G.~Barles, H.~M. Soner, and P.~E. Souganidis.
\newblock Front propagation and phase field theory.
\newblock {\em SIAM J. Control Optim.}, 31(2):439--469, 1993.

\bibitem{BeNo97}
G.~Bellettini and M.~Novaga.
\newblock Minimal barriers for geometric evolutions.
\newblock {\em J. Differential Equations}, 139(1):76--103, 1997.

\bibitem{BePa95}
G.~Bellettini and M.~Paolini.
\newblock Some results on minimal barriers in the sense of {D}e {G}iorgi
  applied to driven motion by mean curvature.
\newblock {\em Rend. Accad. Naz. Sci. XL Mem. Mat. Appl. (5)}, 19:43--67, 1995.

\bibitem{Bell13}
Giovanni Bellettini.
\newblock {\em Lecture notes on mean curvature flow, barriers and singular
  perturbations}, volume~12 of {\em Appunti. Scuola Normale Superiore di Pisa
  (Nuova Serie) [Lecture Notes. Scuola Normale Superiore di Pisa (New
  Series)]}.
\newblock Edizioni della Normale, Pisa, 2013.

\bibitem{Brak78}
Kenneth~A. Brakke.
\newblock {\em The motion of a surface by its mean curvature}, volume~20 of
  {\em Mathematical Notes}.
\newblock Princeton University Press, Princeton, N.J., 1978.

\bibitem{b1} Zdzislaw Brze\'zniak.
\newblock On stochastic convolution in Banach spaces and applications.
\newblock {\em Stoch. Stoch. Rep.} 61:245--295, 1997.

\bibitem{MR1920103}
Rainer Buckdahn and Jin Ma.
\newblock Pathwise stochastic {T}aylor expansions and stochastic viscosity
  solutions for fully nonlinear stochastic {PDE}s.
\newblock {\em Ann. Probab.}, 30(3):1131--1171, 2002.

\bibitem{MR2765508}
Michael Caruana, Peter~K. Friz, and Harald Oberhauser.
\newblock A (rough) pathwise approach to a class of non-linear stochastic
  partial differential equations.
\newblock {\em Ann. Inst. H. Poincar\'e Anal. Non Lin\'eaire}, 28(1):27--46,
  2011.

\bibitem{ChGG91}
Y.~G. Chen, Y.~Giga, and S.~Goto.
\newblock {\em Uniqueness and existence of viscosity solutions of generalized
  mean curvature flow equations}.
\newblock J. Diff. Geom. 33, 749-786, 1991.

\bibitem{daprato} G. Da Prato, J. Zabczyk.
\newblock  \textit{Stochastic Equations in Infinite Dimensions.}
\newblock Encyclopedia Math. Appl., vol. 44, Cambridge University Press, Cambridge, 1992. 

\bibitem{dehovo} A. Debussche, M. Hofmanov\'a, J. Vovelle. Degenerate parabolic stochastic partial differential equations: Quasilinear case.
\newblock {\em ArXiv e-prints}, 2013.

\bibitem{MR1867935}
Nicolas Dirr, Stephan Luckhaus, and Matteo Novaga.
\newblock A stochastic selection principle in case of fattening for curvature
  flow.
\newblock {\em Calc. Var. Partial Differential Equations}, 13(4):405--425,
  2001.

\bibitem{Ecke04}
Klaus Ecker.
\newblock {\em Regularity theory for mean curvature flow}.
\newblock Progress in Nonlinear Differential Equations and their Applications,
  57. Birkh\"auser Boston Inc., Boston, MA, 2004.

\bibitem{EcHu91}
Klaus Ecker and Gerhard Huisken.
\newblock Interior estimates for hypersurfaces moving by mean curvature.
\newblock {\em Invent. Math.}, 105(3):547--569, 1991.

\bibitem{MR2888287}
Abdelhadi Es-Sarhir and Max-K. von Renesse.
\newblock Ergodicity of stochastic curve shortening flow in the plane.
\newblock {\em SIAM J. Math. Anal.}, 44(1):224--244, 2012.

\bibitem{EvSS92}
Lawrence~C. Evans, Halil~Mete Soner, and Panagiotis~E. Souganidis.
\newblock Phase transitions and generalized motion by mean curvature.
\newblock {\em Comm. Pure Appl. Math.}, 45(9):1097--1123, 1992.

\bibitem{EvSp91}
L.~C. Evans and J.~Spruck.
\newblock Motion of level sets by mean curvature. {I}.
\newblock {\em J. Differential Geom.}, 33(3):635--681, 1991.

\bibitem{EvSp92}
L.~C. Evans and J.~Spruck.
\newblock Motion of level sets by mean curvature. {II}.
\newblock {\em Trans. Amer. Math. Soc.}, 330(1):321--332, 1992.

\bibitem{EvSp92a}
L.~C. Evans and J.~Spruck.
\newblock Motion of level sets by mean curvature. {III}.
\newblock {\em J. Geom. Anal.}, 2(2):121--150, 1992.



\bibitem{EvSp95}
Lawrence~C. Evans and Joel Spruck.
\newblock Motion of level sets by mean curvature. {IV}.
\newblock {\em J. Geom. Anal.}, 5(1):77--114, 1995.

\bibitem{MR3249580}
Xiaobing Feng, Yukun Li, and Andreas Prohl.
\newblock Finite element approximations of the stochastic mean curvature flow
  of planar curves of graphs.
\newblock {\em Stoch. Partial Differ. Equ. Anal. Comput.}, 2(1):54--83, 2014.

\bibitem{fland} F. Flandoli, D. G\c{a}tarek. Martingale and stationary solutions for stochastic Navier-Stokes equations.
{\em Probab. Theory Related Fields} 102(3):367-391, 1995.

\bibitem{MR3152786}
Peter Friz and Harald Oberhauser.
\newblock Rough path stability of (semi-)linear {SPDE}s.
\newblock {\em Probab. Theory Related Fields}, 158(1-2):401--434, 2014.

\bibitem{MR1337253}
T.~Funaki.
\newblock The scaling limit for a stochastic {PDE} and the separation of
  phases.
\newblock {\em Probab. Theory Related Fields}, 102(2):221--288, 1995.

\bibitem{2014arXiv1405.5866G}
B.~{Gess} and M.~{R{\"o}ckner}.
\newblock {Stochastic variational inequalities and regularity for degenerate
  stochastic partial differential equations}.
\newblock {\em ArXiv e-prints}, May 2014.

\bibitem{hof} M. Hofmanov\'a, Degenerate parabolic stochastic partial differential equations, Stoch. Pr. Appl. 123 (12) (2013) 4294-4336.

\bibitem{hofse} M. Hofmanov\'a, J. Seidler.
On weak solutions of stochastic differential equations.
{\em Stoch. Anal. Appl.} 30(1):100--121, 2012.

\bibitem{Huis90}
Gerhard Huisken.
\newblock Asymptotic behavior for singularities of the mean curvature flow.
\newblock {\em J. Differential Geom.}, 31(1):285--299, 1990.

\bibitem{Ilma93}
Tom Ilmanen.
\newblock Convergence of the {A}llen-{C}ahn equation to {B}rakke's motion by
  mean curvature.
\newblock {\em J. Differential Geom.}, 38(2):417--461, 1993.

\bibitem{Ilma94}
Tom Ilmanen.
\newblock Elliptic regularization and partial regularity for motion by mean
  curvature.
\newblock {\em Mem. Amer. Math. Soc.}, 108, 1994.

\bibitem{jakubow} A. Jakubowski.
The almost sure Skorokhod representation for subsequences in nonmetric spaces.
{\em Teor. Veroyatnost. i Primenen} 42(1):209--216, 1997); translation in {\em Theory Probab. Appl.} 42(1):167--174, 1998.

\bibitem{citeulike:2163102}
Kyozi Kawasaki and Takao Ohta.
\newblock Kinetic drumhead model of interface. {I}.
\newblock {\em Progress of Theoretical Physics}, 67(1):147--163, 1982.

\bibitem{MR2284215}
Robert Kohn, Felix Otto, Maria~G. Reznikoff, and Eric Vanden-Eijnden.
\newblock Action minimization and sharp-interface limits for the stochastic
  {A}llen-{C}ahn equation.
\newblock {\em Comm. Pure Appl. Math.}, 60(3):393--438, 2007.

\bibitem{krylov} I. Gy\"{o}ngy, N. Krylov.
Existence of strong solutions for It\^o's stochastic equations via approximations.
\newblock {\em Probab. Theory Related Fields} 105(2):143--158, 1996.



\bibitem{LiSo02}
P.-L. Lions and P.~E. Souganidis.
\newblock Viscosity solutions of fully nonlinear stochastic partial
  differential equations.
\newblock {\em S\=urikaisekikenky\=usho K\=oky\=uroku}, (1287):58--65, 2002.
\newblock Viscosity solutions of differential equations and related topics
  (Japanese) (Kyoto, 2001).

\bibitem{MR1647162}
Pierre-Louis Lions and Panagiotis~E. Souganidis.
\newblock Fully nonlinear stochastic partial differential equations.
\newblock {\em C. R. Acad. Sci. Paris S\'er. I Math.}, 326(9):1085--1092, 1998.

\bibitem{MR1659958}
Pierre-Louis Lions and Panagiotis~E. Souganidis.
\newblock Fully nonlinear stochastic partial differential equations: non-smooth
  equations and applications.
\newblock {\em C. R. Acad. Sci. Paris S\'er. I Math.}, 327(8):735--741, 1998.

\bibitem{MR1799099}
Pierre-Louis Lions and Panagiotis~E. Souganidis.
\newblock Fully nonlinear stochastic pde with semilinear stochastic dependence.
\newblock {\em C. R. Acad. Sci. Paris S\'er. I Math.}, 331(8):617--624, 2000.

\bibitem{MR1807189}
Pierre-Louis Lions and Panagiotis~E. Souganidis.
\newblock Uniqueness of weak solutions of fully nonlinear stochastic partial
  differential equations.
\newblock {\em C. R. Acad. Sci. Paris S\'er. I Math.}, 331(10):783--790, 2000.

\bibitem{LuSt95}
Stephan Luckhaus and Thomas Sturzenhecker.
\newblock Implicit time discretization for the mean curvature flow equation.
\newblock {\em Calc. Var. Partial Differential Equations}, 3(2):253--271, 1995.

\bibitem{malek} J. M\'alek, J. Ne\v{c}as, M. Rokyta, M. R\r{u}\v{z}i\v{c}ka.
\newblock {\em Weak and Measure-valued Solutions to Evolutionary PDEs.}
Chapman \& Hall, London, Weinheim, New York, 1996.


\bibitem{Mant11}
Carlo Mantegazza.
\newblock {\em Lecture notes on mean curvature flow}, volume 290 of {\em
  Progress in Mathematics}.
\newblock Birkh\"auser/Springer Basel AG, Basel, 2011.

\bibitem{on1} Z. Brze\'zniak, M. Ondrej\'at, Strong solutions to stochastic wave equations with values in Riemannian manifolds, J. Funct. Anal. 253 (2007) 449-481.

\bibitem{pr07} C. Pr\'ev\^ot, M. R\"ockner, A concise course on stochastic partial differential equations, vol. 1905 of Lecture Notes in Math., Springer, Berlin, 2007.

\bibitem{zbMATH06217659}
Matthias {R\"oger} and Hendrik {Weber}.
\newblock {Tightness for a stochastic Allen-Cahn equation.}
\newblock {\em {Stoch. Partial Differ. Equ., Anal. Comput.}}, 1(1):175--203,
  2013.
  
\bibitem{simon} J. Simon, Compact sets in the space $L^p(0,T;B)$, Ann. Mat. Pura Appl. (4) 146 (1987), 65-96.

\bibitem{MR2037245}
P.~E. Souganidis and N.~K. Yip.
\newblock Uniqueness of motion by mean curvature perturbed by stochastic noise.
\newblock {\em Ann. Inst. H. Poincar\'e Anal. Non Lin\'eaire}, 21(1):1--23,
  2004.

\bibitem{MR2642385}
Hendrik Weber.
\newblock Sharp interface limit for invariant measures of a stochastic
  {A}llen-{C}ahn equation.
\newblock {\em Comm. Pure Appl. Math.}, 63(8):1071--1109, 2010.

\bibitem{MR1656479}
Nung~Kwan Yip.
\newblock Stochastic motion by mean curvature.
\newblock {\em Arch. Rational Mech. Anal.}, 144(4):313--355, 1998.

\bibitem{MR1931534}
Xi-Ping Zhu.
\newblock {\em Lectures on mean curvature flows}, volume~32 of {\em AMS/IP
  Studies in Advanced Mathematics}.
\newblock American Mathematical Society, Providence, RI, 2002.

\end{thebibliography}

\end{document}